\documentclass[11pt]{amsart}
\usepackage{amsthm,amsmath,amsxtra,amscd,amssymb,xypic,color}
\usepackage{mathtools}
\usepackage{enumerate}
\usepackage{multirow}
\usepackage{tikz}
\usepackage{hyperref}
\usepackage{graphicx,transparent}
\usetikzlibrary{calc,matrix,arrows,shapes,decorations.pathmorphing,decorations.markings,decorations.pathreplacing}

\DeclareRobustCommand{\SkipTocEntry}[9]{}

\usepackage[style=alphabetic,
          isbn=false,
          doi=false,
          url=false,
          backend=bibtex, maxnames = 5
          ]{biblatex}
\defbibheading{bibliography}[\bibname]{%
\addcontentsline{toc}{part}{References}%
\addtocontents{toc}{\SkipTocEntry}%
\section*{References}%
}
\DeclareFieldFormat[article]{title}{{\it #1}} 
\DeclareFieldFormat[incollection]{title}{{\it #1}}
\renewbibmacro{in:}{
  \ifentrytype{article}{}{\printtext{\bibstring{in}\intitlepunct}}}

\numberwithin{equation}{section}

\newcommand{\be}{\beta}

\renewcommand{\Im}{{\operatorname{Im}}}
\newcommand{\ord}{\operatorname{ord}\nolimits}
\newcommand{\CC}{{\mathbb{C}}}

\newcommand{\PP}{{\mathbb{P}}}

\newcommand{\RR}{{\mathbb{R}}}
\newcommand{\ZZ}{{\mathbb{Z}}}

\newcommand{\VV}{{\mathbb{V}}}
\newcommand{\calO}{{\mathcal O}}

\newcommand{\calC}{{\mathcal C}}
\newcommand{\calL}{{\mathcal L}}

\newcommand{\calQ}{{\mathcal Q}}

\newcommand{\calX}{{\mathcal X}}
\newcommand{\calY}{{\mathcal Y}}

\newcommand{\calD}{{\mathcal D}}

\newcommand{\op}{\operatorname}

\newcommand{\SL}{\op{SL}}

\newcommand{\ab}{\operatorname{ab}}
\newcommand{\nonab}{\operatorname{non-ab}}

\newcommand{\proj}{{\mathbb P}}

\newcommand\Res{\operatorname{Res}}

\newcommand{\barmoduli}[1][g]{{\overline{\mathcal M}}_{#1}}

\newcommand{\moduli}[1][g]{{\mathcal M}_{#1}}

\newcommand{\omoduli}[1][g]{{\Omega\mathcal M}_{#1}}

\newcommand{\oG}{\overline{\Gamma}}

 \newcommand{\divisor}[1]{{\rm div }\left( #1 \right)}

\newcommand{\rom}[1]{\textup{\uppercase\expandafter{\romannumeral#1}}}

\theoremstyle{plain}
\newtheorem{thm}{Theorem}[section]
\newtheorem{lm}[thm]{Lemma}
\newtheorem{prop}[thm]{Proposition}
\newtheorem{cor}[thm]{Corollary}

\theoremstyle{definition}
\newtheorem{df}[thm]{Definition}

\newtheorem{exa}[thm]{Example}

\def\be{\begin{equation}}   \def\ee{\end{equation}}     \def\bes{\begin{equation*}}    \def\ees{\end{equation*}}
\def\ba{\be\begin{aligned}} \def\ea{\end{aligned}\ee}   \def\bas{\bes\begin{aligned}}  \def\eas{\end{aligned}\ees}
\def\={\;=\;}  \def\+{\,+\,} 

\bibliography{plumb_biblio}

\newcommand{\whX}{\widehat{X}}
\newcommand{\whY}{\widehat{Y}}
\newcommand{\whP}{\widehat{P}}
\newcommand{\whZ}{\widehat{Z}}

\newcommand{\whg}{\widehat{g}}
\newcommand{\whmu}{\widehat{\mu}}
\newcommand{\whLa}{\widehat{\Lambda}}
\newcommand{\whq}{\widehat{q}}
\newcommand{\wh}[1]{{\widehat{#1}}}

\newcommand{\divxi}{E}
\newcommand{\divxired}{E_{\rm red}}
\newcommand{\komoduli}[1][g]{{\Omega^k\mathcal M}_{#1}}
\newcommand{\komoduliab}[1][\mu]{\Omega^{k,{\operatorname{ab}}}\mathcal{M}_{g}(#1)}
\newcommand{\komodulinoab}[1][\mu]{\Omega^{k,{\operatorname{non-ab}}}\mathcal{M}_{g}(#1)}
\newcommand{\tkdab}[1][\overline{\Gamma}]{\mathfrak{W}^{k, \operatorname{ab}}(#1)}
\newcommand{\tkdnoab}[1][\overline{\Gamma}]{\mathfrak{W}^{k, \operatorname{non-ab}}(#1)}
\newcommand{\tdab}[1][\overline{\Gamma}]{\mathfrak{W}^{\operatorname{ab}}(#1)}

\newcommand{\kobarmoduli}[1][g]{{\Omega^k\overline{\mathcal M}}_{#1}}
\newcommand{\kobarmodulin}[1][g,n]{{\Omega^k\overline{\mathcal M}}_{#1}}
\newcommand{\kivc}[2][g,n]{{\PP\Omega^k{\overline{\mathcal M}}_{#1}^{{\rm inc}}(#2)}}
\newcommand{\Resk}[1][k]{\Res^{#1}}

\newcommand{\cplxliek}{\mathcal{C}^{\bullet}(\Lied)}
\newcommand{\Lied}[1][\xi]{\mathcal{L}_{#1}}
\newcommand{\whomega}{\wh\omega}

\newcommand{\whgrc}{${\rm \widehat{g}}$lobal residue condition}

\newcommand{\hk}{k}

\begin{document}
\title[Strata of $\protect\hk$-differentials]{Strata of $k$-differentials}

\author[M. Bainbridge]{Matt Bainbridge}
\email{mabainbr@indiana.edu}
\address{Department of Mathematics, Indiana University, Bloomington, IN 47405, USA}

\author[D. Chen]{Dawei Chen}
\email{dawei.chen@bc.edu}
\address{Department of Mathematics, Boston College, Chestnut Hill, MA 02467, USA}

\author[Q. Gendron]{Quentin Gendron}
\email{gendron@mpim-bonn.mpg.de}
\address{Institut f\"ur algebraische Geometrie, Leibniz Universit\"at Hannover, Welfengarten 1,
30167 Hannover, Germany}
\curraddr{Max Planck institut, Vivatgasse 7, 53111 Bonn, Germany}

\author[S. Grushevsky]{Samuel Grushevsky}
\email{sam@math.stonybrook.edu}
\address{Mathematics Department, Stony Brook University,
Stony Brook, NY 11794-3651, USA}

\author[M. M\"oller]{Martin M\"oller}
\email{moeller@math.uni-frankfurt.de}
\address{Institut f\"ur Mathematik, Goethe-Universit\"at Frankfurt, Robert-Mayer-Str. 6-8,
60325 Frankfurt am Main, Germany}

\thanks{Research of the second author is supported in part by the National Science Foundation under the CAREER grant DMS-13-50396. Research of the fourth author was supported in part by the National Science Foundation under the grant DMS-15-01265, and by a Simons Fellowship in Mathematics (Simons Foundation grant \#341858 to Samuel Grushevsky)}

\begin{abstract}
A $k$-differential on a Riemann surface is a section of the $k$-th power of the canonical line bundle. Loci of $k$-differentials with prescribed number and multiplicities of zeros and poles form a natural stratification of  the moduli space of $k$-differentials. In this paper we give a complete description for the compactification of the strata of $k$-differentials  in terms of pointed stable $k$-differentials, for all $k$. The upshot is a global $k$-residue condition that can also be reformulated in terms of admissible covers of stable curves. Moreover, we study properties of $k$-differentials regarding their deformations, residues, and flat geometric structure.
\end{abstract}

\date{\today}

\maketitle
\tableofcontents

\section{Introduction}

Let $X$ be a smooth complex curve of genus $g$, and let $K_X$ be the canonical line bundle on~$X$. For a positive integer $k$, a non-zero $k$-differential $\xi$ on $X$ is a non-trivial section of $K_X^k$, namely,~$\xi$ can be written locally as $f(z) (dz)^k$, where $f(z)$ is a (possibly meromorphic) function of a local coordinate $z$, and the expression is invariant under change of coordinates. Abelian and quadratic differentials, corresponding to the cases $k=1$ and $k=2$ respectively, exhibit fascinating geometry that arises from associated flat structure and $\SL_2(\RR)$-action --- we refer to the surveys \cite{zorich, wrightSurvey, Bootcamp} for these topics and recent advances.
\par
Let $\mu = (m_1, \ldots, m_n)$ be a tuple of integers such that their sum is $k(2g-2)$ --- we call $\mu$ a partition of $k (2g-2)$ (allowing non-positive entries). If a $k$-differential $\xi$ has exactly~$n$ zeros and poles of respective orders $m_1, \ldots, m_n$, we say that $\xi$ is of type $\mu$. Denote by $\komoduli$ the $k$-th Hodge bundle over $\moduli[g]$ whose fiber over a smooth genus $g$ curve $X$ is the space $H^0(X, K_X^{k})$ of holomorphic $k$ differentials on $X$. The spaces of holomorphic $k$-differentials of all possible zero types form a stratification of $\komoduli$ (and for the meromorphic case one needs to twist by the polar part to define the $k$-th Hodge bundle, see Section~\ref{subsec:stablek} for more details). In this sense, denote by $\komoduli(\mu)$ the stratum parameterizing $k$-differentials of type $\mu$. More precisely, $\komoduli(\mu)$ parametrizes $(X, \xi, z_1, \ldots, z_n)$ where $X$ is a smooth curve of genus $g$, $z_1, \ldots, z_n$ are distinct marked points, and $\xi$ is a $k$-differential with zeros and poles at $z_i$ such that $\ord_{z_i} \xi = m_i$. In this paper we study the geometry of $\komoduli(\mu)$, and in particular, its compactification. Note that for $X$ smooth, $\xi$ determines its zeros and poles, but we mark these points as it will help us study degenerations of $\xi$.
\par
There are similarities, but also fundamental
differences, between the cases $k=1$, $k=2$, and $k \geq 3$. Since the
$k$-th roots of unity are real only for $k \leq 2$, precisely in these
cases the $\SL_2(\RR)$-action on the strata is well-defined by varying the corresponding flat structure. Technically, however, the case $k=1$ is further different from $k \geq 2$. First, the dimension formula is different in the case of holomorphic abelian differentials compared to all the other cases (see Theorem~\ref{thm:introdim} below).
Moreover, the global residue condition, which is the upshot in the characterization of the strata compactification, differs significantly for $k=1$ and for any $k \geq 2$ (compare Definitions~\ref{def:twistedAbType} and~\ref{def:GRCk}).
\par
Recall that period coordinates provide local coordinates for the strata of abelian and quadratic differentials \cite{HubbardMasur}. In particular, the existence of period coordinates implies the smoothness of the strata in these cases, and gives the dimension of the strata. We will show that period coordinates can also be used for the strata of $k$-differentials for all $k$. Our first result is as follows.
\par
\begin{thm} \label{thm:introdim}
Every connected component of the stratum $\komoduli(\mu)$ is a smooth orbifold. If the component parameterizes $k$-th powers of holomorphic abelian differentials, it has dimension $2g-1+n$. Otherwise it has dimension $2g-2 + n$.
\end{thm}
\par
This theorem in particular implies that every connected component of $\komoduli(\mu)$ is irreducible.
For $k=1$, the strata of abelian differentials can have up to three connected components, caused by hyperelliptic and spin structures, see~\cite{kozo1}. We remark that for $k > 1$, the strata of $k$-differentials can have more connected components due to another reason --- for a common divisor~$d$ of $m_1, \ldots, m_n$ and $k$, the locus of $d$-th powers of $(k/d)$-differentials of type $\mu/d$ provides connected components for $\komoduli(\mu)$
(see Section~\ref{sec:DefOfMeroDiff} for more details).
\par
For the question of compactifying $\komoduli(\mu)$, we follow the same setup as
in the case of $k=1$ treated in~\cite{plumb}. First consider the case when $\mu = (m_1,\ldots,m_n)$ is of holomorphic type, i.e., when all $m_i \geq 0$.
Consider the pullback $\komoduli[g,n]$ of $\komoduli$ from $\moduli[g]$ to the moduli space $\moduli[g,n]$ of $n$-pointed genus $g$ curves.
Then $\komoduli[g,n]$ extends to a vector bundle
$\kobarmoduli[g,n]$ over the Deligne-Mumford compactification $\barmoduli[g,n]$, which parameterizes {\em pointed stable $k$-differentials}. More precisely, let $f: \calC \to \barmoduli[g,n]$ be the universal curve and $\omega_f$ the relative dualizing sheaf. Then we have $\kobarmoduli[g,n] = f_{*} (\omega_f^{ k})$, which is a direct generalization of the total space of the Hodge bundle in the case $k = 1$.  Taking the closure of the projectivized stratum $\PP\komoduli[g](\mu)$ in the projective bundle $\PP\kobarmoduli[g,n]$, parameterizing $k$-differentials up to scaling by a non-zero complex number, we thus obtain a natural compactification $\kivc{\mu}$, called the {\em incidence variety compactification} of the projectivized stratum $\proj\komoduli[g](\mu)$. If some $m_i < 0$, we twist $\omega_f^{k}$ by the polar part of $\mu$ and then carry out the same construction (see Section~\ref{sec:stablek} for more details).
\par
As in the case $k=1$ treated in \cite{plumb}, the idea of the incidence
variety compactification is to record the limit location of zeros and
poles on those components where the limit differential vanishes identically
and to keep track of
the relative scale of those components where the limit differential
is non-zero.
\par
Our characterization of the boundary points of $\kivc{\mu}$ is in terms of
twisted $k$-differentials compatible with a full order (also called a level graph)
on the set of irreducible components of a pointed stable curve. Simply speaking, a twisted $k$-differential
$\eta = \{\eta_v\}$ of type~$\mu$ on a nodal curve~$X$ is a collection
of $k$-differentials on the irreducible components~$X_v$ of~$X$ satisfying the conditions:
(0) vanishing on the smooth locus of $X$ as prescribed
by $\mu$, (1) matching zero and pole orders at each node, and (2) matching residues
at poles of order~$k$ (see Definition~\ref{df:twdk} for more details). This definition of twisted $k$-differentials
is a direct generalization of the case $k=1$ in \cite{plumb}.
\par
To state the extra compatibility conditions that characterize twisted $k$-differentials that lie in the boundary of
$\kivc{\mu}$, we first recall how it works for abelian differentials, i.e., for the case $k=1$.
\par
Let $\Gamma$ be the dual graph of a nodal curve $X$. A {\em full order} on the irreducible components of~$X$ is a relation $\succcurlyeq$ on the set $V$ of vertices of $\Gamma$ that is reflexive, transitive and such that for any two vertices $v_{1}$ and $v_{2}$ at least one of the statements $v_{1}\succcurlyeq v_{2}$ or $v_{2}\succcurlyeq v_{1}$ holds.
Remark that any map $\ell:V\to\RR$ assigning real numbers to vertices of $\Gamma$ defines a full
order on $\Gamma$ by setting $v_1 \succcurlyeq v_2$ if and only if $\ell(v_1) \geq \ell(v_2)$.
We call a dual graph $\Gamma$ equipped with a full order on its vertices a {\em level graph},
and denote it by~$\overline{\Gamma}$. An edge is called horizontal if it joins two vertices at the same level, and called vertical otherwise.
If $e$ is an edge between $X_{v_{1}}\succ X_{v_{2}}$, then the corresponding points that are glued to form the node $e$ are denoted by $q_{e}^{+}$ on $X_{v_{1}}$ and $q_{e}^{-}$ on $X_{v_{2}}$. Finally, we denote by $X_{>L}$
the subcurve of $X$ consisting of all irreducible components $X_{v}$ such that $\ell(v)>L$. See Section~\ref{subset:level} for more details of these notions. We now recall the definition of twisted abelian differentials ($k=1$) compatible with a level graph.
\par
\begin{df}[\cite{plumb}]\label{def:twistedAbType}
Let $(X, z_1, \ldots, z_n)$ be a stable $n$-pointed  curve with dual graph $\Gamma$, and let~$\overline\Gamma$ be a full order on $\Gamma$.
A twisted {\em abelian} differential $\eta$ of type $\mu$ on $X$ is called {\em compatible with~$\overline\Gamma$} if it satisfies the following conditions:
\begin{itemize}
\item[(3)]{\bf (Partial order)} If a node identifies $q_1 \in X_{v_1}$ with $q_2 \in X_{v_2}$, then $v_1\succcurlyeq  v_2$ if and only if $\ord_{q_1} \eta_{v_1}\ge -1$. Moreover,  $v_1\asymp v_2$ if and only if
$\ord_{q_1} \eta_{v_1} = -1$.
\item[(4-ab)] {\bf (Global residue condition for abelian differentials)} For every level $L$
and every connected component~$Y$ of $X_{>L}$ that does not
contain a marked point with a prescribed pole (i.e., there is no $z_i\in Y$ with $m_i<0$) the following condition holds. Let
$\{q_1,\ldots,q_b\}$ denote the set of all nodes where~$Y$ intersects $X_{=L}$. Then
$$ \sum_{j=1}^b\Res_{q_j^-}\eta_{v^-(q_j)}\=0,$$
where we recall that $q_j^-\in X_{=L}$ and $v^-(q_j)$ is a vertex of $\overline\Gamma_{=L}$.
\end{itemize}
\end{df}
The main result of~\cite{plumb} characterizes the boundary of the incidence variety compactification of strata of abelian differentials in terms of
twisted abelian differentials compatible with level graphs.

For general $k$, the first idea is to lift a twisted $k$-differential via a certain covering construction, and reduce the problem to the case of abelian differentials on the covering curve. While this is easy to state, the resulting  condition $(\widehat 4)$ (see Definition~\ref{def:hatGRC}) formulated this way depends on the choice of the cover, and is thus less direct to check in an explicit example. We thus proceed to give an alternative characterization of the closure via a condition $(4)$ (see Definition~\ref{def:GRCk})
imposed directly on the twisted $k$-differential, without passing to a cover. The resulting statement is more elaborate, but is verifiable directly on the curve without investigating the set of all suitable covers.

\smallskip
For the approach via the covering construction, we proceed as follows.
For a $k$-differential $\xi$ on a smooth connected curve $X$, there exists a \emph{canonical cover} $\pi: \wh{X}\to X$, cyclic of degree $k$, such that $\pi^{*}\xi = \omega^k$ for an abelian differential $\omega$ on $\whX$ and such that $\tau^{*} \omega = \zeta \omega$, where $\tau$ is a deck transformation of order $k$ and $\zeta$ is a primitive
$k$-th root of unity. If $X$ is nodal and $\eta$ is a twisted $k$-differential on $X$, then one can first carry out the canonical covering construction over each irreducible component of $X$, and then glue them to form an admissible cover $\pi: \wh{X} \to X$ of a nodal curve (in the sense of admissible covers, see \cite[Paragraph~3.G]{hamobook}). If there exists a twisted abelian differential $\wh{\omega}$ on $\wh{X}$ such that $\pi^{*}\eta = \wh{\omega}^k$ and $\tau^{*} \wh{\omega} = \zeta \wh{\omega}$, we say that $(\pi: \wh{X} \to X,\wh{\omega})$ is a \emph{normalized cover}. See Sections~\ref{subsec:canonical} and~\ref{sec:lift} for more details on canonical covers of smooth curves and normalized covers of nodal curves, respectively.
\par
Given an admissible cover $\pi: \wh{X} \to X$ with dual graphs $\wh\Gamma$ and $\Gamma$, a full order $\overline{\Gamma}$ can be lifted via $\pi$ to define a full order $\wh{\overline{\Gamma}}$, which we call the \emph{lifted level graph}. Now we can formulate the conditions needed to describe the closure of the stratum of $k$-differentials.
\begin{df}
\label{def:hatGRC}
Let $(X, z_1, \ldots, z_n)\in\barmoduli[g,n]$ be a pointed stable curve and let
$\overline\Gamma$ be a level graph on~$X$.
A twisted $k$-differential $\eta$ of type $\mu$ on $X$ is called {\em compatible with~$\overline\Gamma$} if it satisfies the following conditions:
\begin{itemize}
\item[(3)]{\bf (Partial order)} If a node of $X$  identifies $q_1 \in X_{v_1}$ with $q_2 \in X_{v_2}$, then $v_1\succcurlyeq  v_2$ if and only if $\ord_{q_1} \eta_{v_1}\ge -k$. Moreover,  $v_1\asymp v_2$ if and only if $\ord_{q_1} \eta_{v_1} = -k$.
\item[($\widehat{4}$)]{\bf ($\widehat{\mbox{G}}$lobal residue condition)} There exists a normalized cover $(\whX, \whomega)$ of $(X,\eta)$ such that the twisted abelian differential $\whomega$ satisfies the global residue condition (4-ab) in Definition~\ref{def:twistedAbType} with respect to the lifted level graph~$\wh{\overline\Gamma}$.
\end{itemize}
\end{df}
The above definition of compatibility is in terms of the existence of a normalized cover satisfying certain conditions. By a graph-theoretic argument studying the combinatorics of all possible normalized covers, we will show that it is equivalent to a condition phrased purely in terms of the twisted $k$-differential $\eta$.
\par
In order to state this alternative condition of compatibility, we need the notion of the $k$-residue $\Res^k_z \xi$ for a $k$-differential~$\xi$ with a pole at $z$ of order being an integer multiple of $k$. This $k$-residue is defined to be the $k$-th power of the coefficient of the degree $(-1)$ term of a meromorphic abelian differential~$\omega$ at $z$ such that locally $\xi=\omega^{k}$ (see Proposition~\ref{prop:standard_coordinates}). Note that unless the order of the pole of $\xi$ at $z$ is equal to $k$, the $k$-residue does not have to be equal to the coefficient of $z^{-k}(dz)^k$ in the Laurent expansion of $\xi$ at~$q$. Since such a meromorphic abelian differential is only defined up to multiplication by a $k$-th root of unity, there will be some symmetrization taking place, and hence we introduce the following symmetric polynomial $P_{n,k}$ in $n$ variables
\begin{equation}
\label{eq:P}
P_{n,k}(R_1,\ldots,R_n):=
\prod_{\lbrace (r_1,\ldots,r_n) | r_i^k=R_i\rbrace}\sum_{i=1}^n r_i,
\end{equation}
where the product is taken over all $n$-tuples of complex numbers $\lbrace r_1, \ldots, r_n\rbrace$ such that $r_i^k=R_i$ for all $i$. As $P_{n,k}$ is symmetric with respect to the $k$-th roots of $R_i$, it is indeed a polynomial in~$R_i$.
\par
We are now ready to state the global $k$-residue condition. In order to understand it better, we first point out that the global residue condition ($4$-ab) imposed by the preimage of a connected component $Y$ of $X_{>L}$ at level $L$ is automatically satisfied if the normalized cover above $Y$ has fewer than $k$ components (see Proposition~\ref{prop:fewercomponents}). This can happen in two cases. The first case is when the restriction of $\eta$ to any irreducible component $X_{v}$ of $Y$ is not the $k$-th power of an abelian differential, which leads to condition ii) below. The other case is when the restriction of $\eta$ on every irreducible component of $Y$ is the $k$-th power of an abelian differential but it is possible to make some ``criss-cross" when identifying the nodes of the cover, which leads to conditions iii) and iv) below (see Example~\ref{ex:GRCdependLift} for an illustration of ``criss-cross").
\begin{df}
\label{def:GRCk}
Let $(X, z_1, \ldots, z_n)\in\barmoduli[g,n]$ be a pointed stable curve,
$\overline\Gamma$ a level graph on~$X$ and $\eta$ a twisted $k$-differential of type $\mu$ on $X$ satisfying condition (3) of Definition~\ref{def:hatGRC}. We then define the following
\begin{itemize}
\item[(4)]{\bf (Global $k$-residue condition)} For every level $L$ and every connected component $Y$ of $\Gamma_{>L}$, one of the following cases holds:
\begin{itemize}
\item[i)] $Y$ contains a marked pole.
\item[ii)] $Y$ contains a vertex $v$ such that $\eta_v$ is not a $k$-th power of a (possibly meromorphic) abelian differential.
  \item[iii)] (``Horizontal criss-cross in $Y$.'') For every vertex $v$ of $Y$ the $k$-differential~$\eta_v$ is the $k$-th power of an abelian differential $\omega_v$. Moreover, for every choice of a collection of $k$-th roots of unity $\lbrace \zeta_v:v\in Y\rbrace$ there exists a horizontal edge $e$ in $Y$ where the differentials $\left\{\zeta_v\omega_v\right\}_{v \in Y}$ do not satisfy the matching residue condition.

  \item[iv)] (``Vertical criss-cross in $Y$.'') For every vertex $v$ of $Y$
 the $k$-differential $\eta_{v}$ is the $k$-th power of  an abelian differential $\omega_{v}$. Moreover,  there exists a level $K>L$  and a  collection of $k$-th roots of unity $\lbrace \zeta_e:e\in E\rbrace$ indexed by the set  $E$  of non-horizontal edges $e$ of $Y$ whose lower end lies in $Y_{=K}$,
such that the following two conditions hold.
First, there exists a directed simple loop $\gamma$ in the dual graph of $Y_{\geq K}$
such that
  \begin{equation}\label{eq:cancellation2}
     \prod_{e\in\gamma\cap E} \zeta_{e}^{\pm 1} \neq 1,
  \end{equation}
where the sign of the exponent is $\pm 1$ according to whether
$\gamma$ passes through $e$ in upward or downward direction.
Second, for every connected component $T$ of $Y_{> K}$ the equation
  \begin{equation}\label{eq:cancellation1}
   \sum_{e\in E_T} \zeta_e\Res_{q_e^-}\omega_{v^-(e)}\=0
  \end{equation}
  holds, where $E_{T}$ is the subset of edges in $E$ such that their top vertices lie in $T$.

\item[v)] (``$Y$ imposes a residue condition.'')
 The $k$-residues at the edges $e_{1},\ldots,e_{N}$ joining $Y$ to $\Gamma_{=L}$ satisfy the equation
  \begin{equation}\label{eq:GRCk}
   P_{N,k}\left(\Resk_{q_{e_{1}}^-}\eta_{v^-(e_{1})},\ldots,\Resk_{q_{e_{N}}^-}\eta_{v^-(e_{N})}\right)\=0,
  \end{equation}
  where $P_{N,k}$ is defined by~\eqref{eq:P}.
 \end{itemize}
\end{itemize}
\end{df}
This seemingly complicated condition (4) is imposed directly on $(X, \eta)$, unlike the condition~$(\widehat 4)$ that is in terms of existence of a suitable normalized cover. We point out that one needs to check iv) or v) only if i), ii) and iii) do not hold, in which case the zero and pole orders of $\eta$ at the nodes of $Y$ are all divisible by $k$, and thus the notion of $k$-residue and consequently the residue condition are not trivial.
\par
We will prove in Section~\ref{sec:admtokGRC} that for a twisted $k$-differential satisfying conditions (0)-(3), conditions $(\widehat{4})$ and (4) are equivalent, i.e.,~that a twisted $k$-differential is compatible with $\overline\Gamma$ if and only if it satisfies conditions (0)-(3) and (4).
\par
Finally we can state our main theorem that describes the incidence variety compactification of the strata of $k$-differentials.
\par
\begin{thm} \label{thm:kmain}
A pointed stable $k$-differential $(X,\xi,z_1,\ldots,z_n)$ is contained in the
incidence variety compactification $\kivc{\mu}$ of a
stratum $\proj\komoduli[g](\mu)$ if and only if the following conditions hold:
\begin{itemize}
\item[(i)] There exists a full order $\overline\Gamma$ on $X$ such that its maxima are the irreducible components $X_v$ of $X$ on which~$\xi$ is not identically zero.
\item[(ii)] There exists a twisted $k$-differential~$\eta$ of type $\mu$ on~$X$,
compatible with $\overline\Gamma$.
\item[(iii)] On every irreducible component $X_v$ where $\xi$ is not identically zero,
$\eta_{v} = \xi|_{X_v}$.
\end{itemize}
\end{thm}
\par
\smallskip
In our future work \cite{BCGGM2} the space of twisted differentials by itself
will play a central role. The following dimension statement is the basic
reason that spaces of twisted differentials (or rather finite covers of them)
can be used as orbifold charts for a smooth compactification of strata
of abelian differentials and $k$-differentials, which will be constructed in~\cite{BCGGM2}.  Let~$\overline{\Gamma}$ be
a level graph and let $h$ be the number of horizontal edges
of~$\overline{\Gamma}$. Let $\komoduliab$ be the union of the components
of $\komoduli(\mu)$ consisting of $k$-th powers of holomorphic abelian
differentials, and $\komodulinoab$ is the union of the other components.
\par
\begin{thm} \label{thm:dimtwd}
The space of twisted abelian differentials of type~$\mu$
compatible with~$\oG$ is either empty or has pure dimension equal
to $\dim \omoduli(\mu) - h$.
\par
For any $k\geq 2$, the space of twisted $k$-differentials of type~$\mu$
compatible with~$\oG$ that are smoothable into $\komodulinoab$
is either empty or has pure dimension equal to  $\dim \komodulinoab - h$.
\end{thm}
\par
The same result holds by replacing ``non-ab''  with ``ab'' in the above theorem, whose proof reduces to the abelian case. See Section~\ref{sec:dimension} for more detailed definitions and statements.
\par
\smallskip
This paper is organized as follows. In Section~\ref{sec:DefOfMeroDiff} we study tangent spaces of the strata of $k$-differentials and prove Theorem~\ref{thm:introdim}. We also establish period coordinates in general and describe
$k$-differentials from the viewpoint of flat geometry. In Section~\ref{sec:stablek} we study the notion of $k$-residues and introduce in detail the ambient space of pointed stable $k$-differentials that we use for the incidence variety compactification. In Sections~\ref{sec:lift}  we prove Theorem~\ref{thm:kmain} under the case when condition ($\widehat{4}$) of Definition~\ref{def:hatGRC} is satisfied. In Section~\ref{sec:admtokGRC} we show the equivalence of condition ($\widehat{4}$) and condition (4) of Definition~\ref{def:GRCk} by a graph-theoretic argument. In Section~\ref{sec:dimension} we prove Theorem~\ref{thm:dimtwd} by a homological argument. Finally in Section~\ref{sec:examples}, we provide some examples that are consequences of the incidence variety compactification, and we also outline a flat geometric approach for proving our main result.
\par
\subsection*{Acknowledgements.}
We thank the {\em Casa Matem\'atica Oaxaca (CMO)}, the {\em Mathematisches Forschungsinstitut Oberwolfach (MFO)},  the {\em Institute for Computational and Experimental Research in Mathematics (ICERM)}, and the {\em Max-Planck Institut
f\"ur Mathematik, Bonn (MPIM)}, where various subsets of the authors met
and substantial progress on this work was made. We are grateful to the
organizers of the conferences for the invitations to participate. We also thank the anonymous referee for many valuable comments and suggestions, in particular leading to the previously missing point (iii) of Definition~\ref{def:GRCk}.



\section{Deformations of $k$-differentials}
\label{sec:DefOfMeroDiff}
The goal of this section is to study the local structure of the strata of $k$-differentials. We identify the tangent spaces of the strata with a hypercohomology group and compute its dimension in Theorem~\ref{thm:DeformkDiff}. As in the well-studied cases of abelian and quadratic differentials (i.e., $k=1$ and $k=2$), in Theorem~\ref{thm:periodCoorkDiff} we exhibit period coordinates for the strata of $k$-differentials. In the course of the proof, we introduce the key notion of canonical covers, which will be used throughout the paper. We also explain the flat geometric meaning of $k$-differentials in terms of $(1/k)$-translation structure, which we will use to present examples.
\par
Let us first introduce the major results of this section. Denote by $\divisor\xi=\sum_{i=1}^n m_{i}z_{i}$ the underlying divisor of a $k$-differential $\xi$ on a smooth curve $X$.
If $E = \sum_j m_j p_j$ is any divisor on $X$ with non-zero integer coefficients $m_j$, we let
$\divxired = \sum_j p_j$ be the corresponding reduced divisor and we decompose it
as $\divxired = Z - P$ with
$$Z \= \sum_{m_j > -k} p_j \quad \text{and} \quad P \=  \sum_{m_j \leq -k} p_j$$
into its ``zero'' and ``pole'' divisor. We stress that the poles of order $m_{i}>-k$ of a $k$-differential $\xi$ contribute to the zero part of $\divisor\xi_{\rm red}$ (in the sense when passing to the canonical $k$-cover). In the result below we show that the tangent space of the stratum at $\xi$
can be identified with the
hypercohomology of the following two-term complex of sheaves in degrees zero and one
\begin{equation}\label{eq:Liegeneral}
 \cplxliek \= \Lied:T_{X}(-\divisor\xi_{\rm red})\to K_{X}^{k}(-\divisor\xi),
\end{equation}
where $\Lied$ is the $k$-Lie derivative associated to~$\xi$, whose definition
we will recall in Section~\ref{subsec:Lie}.
\par
\begin{thm}\label{thm:DeformkDiff}
Let $(X,\xi,z_{1},\ldots,z_{n}) \in  \komoduli(m_1,\ldots,m_n)$ be a (possibly meromorphic)
$k$-differential on a smooth curve $X$. Then the tangent space of the stratum
at $\xi$ can be identified as
$$
T_ {(X,\xi,z_{1},\ldots,z_{n})}\komoduli(m_1,\ldots,m_n)  \=  H^{1}(X, \cplxliek)\,.
$$
Moreover, the connected component of $\komoduli({m_1,\ldots,m_n})$ containing the point $(X,\xi,z_1,\ldots,z_n)$ is a smooth orbifold
of dimension $2g-1+n$ if $\xi$ is the $k$-th power of a holomorphic abelian differential
and of dimension $2g-2 + n$ otherwise.
\end{thm}
\par
As mentioned in the introduction, for a divisor $d$ of $k$, the loci of $d$-th powers of $(k/d)$-differentials can provide connected components
for the strata of $k$-differentials. A $k$-differential is called {\em primitive}, if it is
not the $d$-th power of a $(k/d)$-differential for any divisor $d>1$ of~$k$. For local computations
it suffices to focus on the case of primitive $k$-differentials, because otherwise we may pass to the case of
$(k/d)$-differentials instead.
\par
Given a $k$-differential~$\xi$ on $X$, there exists a canonical cover $\pi: \wh{X} \to X$,
which is cyclic of degree $k$, such that
the pullback of $\xi$ is the $k$-th power of an abelian differential~$\omega$ on~$\wh{X}$.
Moreover, there exists a deck transformation
$\tau$ of $\pi$ of order $k$ such that $\tau^{*} \omega = \zeta \omega$, where $\zeta$ is a chosen primitive $k$-th root of unity (fixed from now on for the rest of the paper). A canonical cover together with such a deck transformation is called {\em normalized}, and denoted by the pair $(\pi, \tau)$.
\par
We remark that $\wh{X}$ is connected if and only if $\xi$ is primitive. In that case, the deck transformation group of $\pi$ is cyclic of order $k$,  generated by $\tau$. If $\xi$ is not primitive, then $\wh{X}$ is disconnected, and $\tau$ only generates a cyclic subgroup of order $k$ in the deck transformation group. The construction of canonical cover and its properties will be explained in Section~\ref{subsec:canonical}.
\par
Let $\wh{E}$ be the $\pi$-pullback of $\divisor\xi$ and we denote the $\pi$-preimage
of the poles~$P$ and zeros~$Z$ by $\wh{P}$ and $\wh{Z}$ respectively. Let
$\wh\imath: \wh{X} \setminus \wh{P} \to \wh{X}$ and
$\wh\jmath: \wh{X} \setminus \wh{E}_{\rm red} \to \wh{X} \setminus \wh{P}$ be the inclusion maps.
We can describe period coordinates for $\komoduli({m_1,...,m_n})$ using the cohomology
on $\wh{X}$ as follows.
\begin{thm}\label{thm:periodCoorkDiff}
Let $\xi$ be a primitive $k$-differential in $\komoduli(m_1,\ldots,m_n)$. The tangent space
$T_ {(X,\xi,z_{1},\ldots,z_{n})}\komoduli(m_1,\ldots,m_n)$ is isomorphic to the $\tau$-eigenspace
of $H^{1}(\whX, \wh{\imath}_*(\wh\jmath_{!}\,\CC))$ for the eigenvalue $\zeta$.
\end{thm}
As a consequence of the above theorem, we exhibit period coordinates geometrically for the strata.
\begin{cor}\label{cor:period}
Let $V_{\zeta}$ be the eigenspace of $H_{1}(\whX \setminus \wh{P}, \wh{Z} ;\CC)$ associated to the
eigenvalue $\zeta$. Locally at the primitive $k$-differential $\xi$, the
stratum $\komoduli(m_1,\ldots,m_n)$ has coordinates given by the periods
$\int_{\gamma_i} \omega$
where $\{\gamma_i \}$ is a basis of $V_\zeta$ and $\pi^{*}\xi = \omega^k$ for an abelian differential $\omega$ on $\whX$.
\end{cor}
We remark that Theorem~\ref{thm:DeformkDiff} is a generalization of \cite[Theorem~2.3]{moellerLin}
and \cite[Proposition~3.1]{mondello} from the case of abelian differentials to $k$-differentials. Recently Mondello~\cite{mondello-personal} and Schmitt~\cite{schmitt-strata} have also independently obtained this result. Theorem~\ref{thm:periodCoorkDiff} generalizes the known case of quadratic differentials, originally studied by Hubbard-Masur~\cite{HubbardMasur} and Veech~\cite{Veech}.

\subsection{The canonical cover}\label{subsec:canonical}
We begin by constructing a cyclic cover of $X$ such that the pullback of $\xi$ is the $k$-th
power of an abelian differential. A reference for this material is \cite[Section~3]{esvivanishing}. We remark that here a cover
means a possibly ramified and disconnected cover.
\par
Recall that $\zeta$ is a fixed primitive $k$-th root of unity. We group the zeros and poles of $\xi$ according to their remainders mod~$k$, so that
$$
\divisor{\xi}\=\sum_{j=0}^{k-1}\sum_{i \in I_j} m_{i,j} z_{i,j} \quad \text{where}
\quad I_{j}\=\left\{i\in \{1,\ldots,n\}\,:\, m_{i}\equiv j \mod k\right\}
$$
and we write the coefficients as $m_{i,j} = k\ell_{i,j} + j $. Next we define a line bundle
\begin{equation}\label{eq:ramdata}
 \calL \= K_X \Biggl(- \sum_j \sum_{i \in I_j} l_{i,j} z_{i,j}\Biggr)
\quad \text{and} \quad D \= \sum_{j=0}^{k-1} \sum_{i \in I_j} jz_{i,j},
\end{equation}
the remainder of $\divisor{\xi}$ mod~$k$. Then $\xi$ can be considered as a section~$s$ of the line
bundle~$\calL^{ k}$, with zero divisor~$D$. The triple $(\calL,s,D)$ is the defining
datum of a degree~$k$ cyclic cover $\pi:\whX\to X$, given by providing
the $\calO_X$-module $\oplus_{i=0}^{k-1} \calL^{-i}$ with an algebra
structure by mapping a section~$\alpha$ of~$\calL^{-k}$ to $s\alpha \in \calO_X$
and normalizing the resulting cover, as  explained e.g., in \cite[Paragraph~3.5]{esvivanishing}.
We say that this cover is the {\em canonical cover of $X$ induced by $\xi$}.
This algebraic viewpoint makes it obvious that the construction of the canonical cover
works for a family of smooth curves $f: \calX \to B$ over some base~$B$,
provided with a family of $k$-differentials, i.e., a section of $f_* (\omega_{f}^{k})$, where
$\omega_f$ is the relative dualizing sheaf of the family.
\par
There is an alternative geometric viewpoint of this construction (see
e.g., \cite{Lanneauhyp} for the case $k=2$). One first
constructs an unramified cover $\whX^0$ of $X^0 = X \setminus \divisor\xi_{\rm red}$ as follows.
Cover~$X^0$ by open sets $U_\alpha$ and take $k$ patches
$U_{\alpha,i}  \cong U_\alpha$ for $i=0,\ldots,k-1$. We provide $U_{\alpha_i}$ with a flat structure
given by the abelian differential $\omega_{\alpha,i} = \zeta^i \, \xi|_{U_\alpha}^{1/k}$, for some choice
of a $k$-th root of $\xi$ over~$U_\alpha$. Note that different choices simply permute the labeling
of the $U_{\alpha,i}$, so that the construction in the end will not depend on these choices.
Whenever $U_\alpha$ and $U_\beta$ intersect in $X^0$, we glue the patches
$U_{\alpha,i}$ and $U_{\beta,j}$ for the indices $i$ and $j$ such that $\varphi_{\alpha\beta}^*
\omega_{\alpha,i} = \omega_{\beta,j}$ to form $\whX^0$, where $\varphi_{\alpha\beta}$ is the transition function from $U_\beta$ to $U_\alpha$. Finally, the unramified cover $\whX^0$ extends uniquely to a (possibly ramified) cover $\pi: \whX\to X$
and the local differentials $\omega_{\alpha,i}$ glue to a global differential $\omega$ on~$\whX$ such that
$\omega^k = \pi^* \xi$.
\par
For both constructions it is clear that $\omega$ is an eigenform of a deck transformation $\tau$ such that
$\tau^{*}\omega = \zeta \omega$, and hence $(\pi, \tau)$ is a normalized cover.
Moreover, if $\whX$ has $r$ connected components, then the restrictions of $\omega$ to two connected components differ by a non-trivial power of $\zeta$.
\par
To compute the divisor of~$\omega$, we let~$r_{i,j} = \gcd (m_{i,j}, k) = \gcd(j,k)$
and $d_{i,j} = k / r_{i,j}$. Then in the geometric construction the fiber of $\pi$
over $z_{i,j}$ consists of~$r_{i,j}$ distinct points $z_{i,j}^{(1)}, \ldots, z_{i,j}^{(r_{i,j})}$,
each with multiplicity~$d_{i,j}$, i.e., the ramification divisor of $\pi$
is $$R=\sum_{i,j,\ell}(d_{i,j}-1)z_{i,j}^{(\ell)}.$$
\par
Now we can summarize the properties of canonical cover.
\par
\begin{prop}\label{prop:cancov} The algebraic and geometric constructions
of a canonical cover give isomorphic degree $k$ cyclic covers
$\pi: \whX\to X$. Moreover, there exists an abelian differential $\omega$ on $\whX$ such
that $\omega^k = \pi^* \xi$, where such $\omega$ is unique up to multiplication by a $k$-th root of unity on each connected component of $\whX$.
\par
The canonical cover is connected if and only if $\xi$ is a
primitive $k$-differential. In this case the genus $\wh g$ of $\wh X$ is
$$ \whg \= 1 + k(g-1) + \frac{1}{2} \Bigl(kn - \sum_{i,j} r_{i,j}\Bigr) $$
and the underlying divisor of $\omega$ is
\begin{equation}\label{eq:divpullback}
\divisor\omega \= \sum_{i,j,\ell} \left(d_{i,j}-1 + \frac{d_{i,j} m_{i,j}}{k}\right)
z_{i,j}^{(\ell)} \= \sum_{i,j,\ell} \left(\frac{k - r_{i,j} + m_{i,j}}{r_{i,j}}\right)
z_{i,j}^{(\ell)}.
\end{equation}
\end{prop}
\begin{proof}
The algebraic construction gives a
section~$\wh{s}$ of $\pi^*\calL$ such that $\wh{s}^k = \pi^{*}s$, and the associated divisor of $\wh{s}$ is
$\sum_{i,j,\ell} \tfrac{j}{r_{i,j}} z_{i,j}^{(\ell)}$, as one can see from the
following interpretation (cf.~\cite[Remark~3.14 b)]{esvivanishing}). If
$\VV(\calL^{-k})$ is the geometric line bundle associated with $\calL^{k}$,
then global sections of $\calL^{k}$ are geometric sections of $\VV(\calL^{-k})$.
There is a natural symmetric power map $\rho: \VV(\calL^{-1}) \to \VV(\calL^{-k})$
of geometric bundles over~$X$,
and $\wh{X} = \rho^{-1}(s(X))$ by construction. Consequently, $\wh{X}$ comes
with a natural map to $\VV(\calL^{-1})$ and hence also with a map to its pullback
to $\wh{X}$. Since $K_{\wh{X}} = \pi^* K_X \otimes \calO_{\wh{X}}(R)$, we can
interpret~$\wh{s}$ as a section $\omega$ of $K_{\wh{X}}$ with $\omega^k = \pi^{*}\xi$.
\par
Taking the $\rho$-preimage of a section of $\VV(\calL^{-k})$ is equivalent to gluing
locally $k$-th roots of the section in the unique consistent way, as in the geometric construction.
This implies that the two constructions agree away from the branch points, and hence also agree across the branch points,
as guaranteed by the uniqueness of a smooth extension.
\par
The connectedness claim is shown in \cite[Lemma~3.15]{esvivanishing}. The expressions of the genus and the divisor of $\pi^* \xi$ follow from the Riemann-Hurwitz formula, since we have already determined the
branching orders of $\pi$. In fact, from the geometric construction of the canonical cover it is obvious that $\divisor\omega = \tfrac1k
\divisor{\pi^* \xi}$, and from the algebraic construction $\divisor\omega$
is equivalently given by $R + \sum_{i,j,\ell} (d_{i,j}\ell_{i,j}+ \tfrac{j}{r_{i,j}}) z_{i,j}^{(\ell)}$.
\par
Finally note that if $\widetilde{\omega}$ is another abelian differential on $\whX$ with $\widetilde{\omega}^k = \pi^{*}\xi$, then
$\widetilde{\omega}/\omega$ is a function on $\whX$, whose $k$-th power is one. Hence this function
is equal to a $k$-th root of unity on every connected component of $\whX$.
\end{proof}
\par

\subsection{The Lie derivative}
\label{subsec:Lie}
The Lie derivative $\mathsf{L}_{v}$ associated to
a vector field $v$ applied to a symmetric tensor~$T$ is defined as $\mathsf{L}_{v} = \tfrac{\partial}{\partial t} \varphi_t^* T|_{t=0}$, where $\varphi_t$ is
the one-parameter group of diffeomorphisms generated by~$v$. Axiomatically one can describe
the collection of maps $\mathsf{L}_{v}: K_X^{k} \to  K_X^{ k}$ for all $k \in \ZZ$
by requiring that
\begin{itemize}
\item[(i)] $\mathsf{L}_{v}$ is the directional derivative for $k=0$.
\item[(ii)] It satisfies the Leibniz rule
$$ \mathsf{L}_{v}(S \otimes T) = \mathsf{L}_{v}(S) \otimes T +  S \otimes \mathsf{L}_{v}(T)$$
for tensor products and the analogous Leibniz rule for contractions.
\item[(iii)] $\mathsf{L}_{v}$ commutes with exterior derivative on functions.
\end{itemize}
\par
We need the Lie derivative with~$v$ as argument instead, that is, we define
\begin{equation}
\Lied:T_{X}(-\divisor\xi_{\rm red})\to K_{X}^{k}(-\divisor\xi), \quad
\Lied(v) \coloneqq \mathsf{L}_{v}(\xi)
\end{equation}
to be the  {\em $k$-Lie derivative associated to $\xi\in H^{0}(X,K_{X}^{k})$}. Using the axioms given above, we find by induction on $k$ that
in the local coordinate $z$ where $\xi=a(z)( d z)^{k}$, the map~$ \Lied$ is given by
\begin{equation}\label{eq:Liederi}
 \Lied\Bigl(b(z)\tfrac{\partial}{\partial z}\Bigr) \= (a'b+kab' )(dz)^k,
\end{equation}
which is the main property we need in the sequel.
\par
\begin{proof}[Proof of Theorem~\ref{thm:DeformkDiff}.] We abbreviate $\divxi = \divisor\xi$.
The first order deformations of a smooth pointed curve $(X,z_1,\ldots, z_n)$
are given by one-cocycles $\left\{ D_{\alpha\beta}\right\}$ with
values in $T_X(-\divxired)$ for a cover of $X$ by open sets~$\{U_\alpha\}$.
Let $\varphi_{\alpha\beta}$ denote the transition function on the
intersection $U_\alpha \cap U_\beta$. We determine over which of these deformations we can extend the
$k$-differential~$\xi$. A deformation of~$\xi$ is locally of the form
\[\widetilde{\xi}_{\alpha}\=\xi_{\alpha}+\epsilon\xi'_{\alpha}\=f_{\alpha}(u_{\alpha})( d u_{\alpha})^{k}+\epsilon g_{\alpha}(u_{\alpha})( d u_{\alpha})^{k},\]
where $\epsilon^2 = 0$, and we require that  $\ord_{p}(\xi'_{\alpha})\geq\ord_{p}(\xi_{\alpha})$
at every point $p\in X_{\alpha}$ for deformations that preserve
the type $(m_1,\ldots,m_n)$ of the $k$-differential.
Moreover, we have
\begin{eqnarray*}
{\varphi^\ast_{\alpha\beta}}(\widetilde{\xi}_{\alpha}) &=&  \left[ f_{\alpha}(u_{\alpha}+\epsilon D_{\alpha\beta}(u_{\alpha}))+\epsilon g_{\alpha}(u_{\alpha})\right]\left( d(u_{\alpha}+\epsilon D_{\alpha\beta}(u_{\alpha}))\right)^{k}  \\
&=& \xi_{\alpha}+\epsilon\left[(f_{\alpha}'D_{\alpha\beta}+kf_{\alpha}D_{\alpha\beta}')( d u_{\alpha})^{k}+\xi'_{\alpha}\right]\\
&=& \xi_{\alpha}+\epsilon\left[\Lied[\xi_{\alpha}](D_{\alpha\beta})+\xi'_{\alpha}\right].
\end{eqnarray*}
Hence the pairs $ \lbrace (D_{\alpha\beta}, \xi'_{\alpha}) \rbrace     \in C^{1}(\{U_{\alpha}\}, T_X(-\divxired))
\oplus C^{0}(\{U_{\alpha}\}, K^{k}_{X}(-\divxi))$ glue to a first order deformation of $(X,\xi,z_{1},\ldots,z_{n})$ if and only if
${\varphi^\ast_{\alpha\beta}} (\widetilde{\xi}_{\alpha})-\widetilde{\xi}_{\beta}=0$, that is $\Lied[\xi_{\alpha}](D_{\alpha\beta})+\xi'_{\alpha}-\xi'_{\beta}=0$ for all $\alpha,\beta$. Moreover, the trivial deformations are given by the trivial deformations of the curve and the deformations such that $\Lied[\xi_{\alpha}](D_{\alpha\beta})+\xi'_{\alpha}=0$ at the level of the differential. Hence the deformation of $(X,\xi)$ are parameterised by the first hypercohomology group of the complex~$\cplxliek$.
\par
The filtration of the complex in the middle of the following diagram
\begin{equation}\label{equation:suiteExacteCourteDeDeformation1}
\begin{xy}
 \xymatrix{
 	0 \ar[r]& 0 \ar[r]\ar[d]&   T_{X}(-\divxired) \ar[r]\ar[d]^{\Lied}& T_{X}(-\divxired) \ar[r]\ar[d]& 0
 \\
 	0 \ar[r]& K_{X}^{k}(-\divxi) \ar[r]&    K_{X}^{k}(-\divxi) \ar[r]& 0 \ar[r]& 0
 	}
\end{xy}
\end{equation}
gives a short exact sequence
of two-term complexes. The associated long exact sequence of cohomology is
\par
\bes
\begin{tikzpicture}[scale=2]
\matrix(m)[matrix of math nodes,column sep=15pt,row sep=15pt]{
  0 & H^{0}(X,T_{X}(-\divxired)) & H^{0}(X, K_{X}^{k}(-\divxi))
& H^{1}(X,\cplxliek ) & \\
    &   H^{1}(X,T_{X}(-\divxired))  & H^{1}(X, K_{X}^{k}(-\divxi))
& H^{2}(X,\cplxliek) & 0\\
};
\draw[->,font=\scriptsize,every node/.style={above},rounded corners]
  (m-1-1) edge (m-1-2)
  (m-1-2) edge  [above] node {$\beta_{0}$} (m-1-3)
  (m-1-3) edge  [above] node {$\gamma_{0}$} (m-1-4)
  (m-1-4.east) --+(5pt,0)|-+(0,-7.5pt)-|([xshift=-5pt]m-2-2.west)--(m-2-2.west)
[below] node {$\alpha_{1}$}
  (m-2-2) edge [above] node {$\beta_{1}$} (m-2-3)
  (m-2-3) edge  [above] node {$\gamma_{1}$} (m-2-4)
  (m-2-4) edge (m-2-5)
;
\end{tikzpicture}
\ees
The Riemann-Roch formula implies that $h^{1}(X,\cplxliek )=2g-2+n+h^{2}(X,\cplxliek )$.
We want to compute $H^{2}(X,\cplxliek)$, which coincides with the cokernel of the map
 $$\beta_{1}:H^{1}(X,T_{X}(-\divxired))  \to H^{1}(X, K_{X}^{k}(-\divxi)).$$
\par
By Serre duality, $H^{2}(X,\cplxliek)$ is dual to the kernel of
the map
$$\beta_{1}^{\vee}=H^{0}(\Lied^{\vee}):H^{0}(X,K_{X}^{1-k}(\divxi))  \to H^{0}(X, K_{X}^{2}(\divxired)),$$
where $\Lied^{\vee}: K_{X}^{1-k}(\divxi)\to K_{X}^{2}(\divxired)$ is the dual map of $\Lied$. We compute this dual map $\Lied^{\vee}$ in local coordinates.
 Let $\tau$ and $v$ be respectively a local section of $K_{X}^{1-k}(\divxi)$ and $T_{X}(-\divxired)$. We fix a local coordinate $z$ on an open subset $U$ such that
 $$\xi=a(z) (dz)^{k}, \quad \tau=c(z) (\partial_{z})^{k-1}, \quad v=b(z)\partial_{z}.$$
 By Equation~\eqref{eq:Liederi} we have
$$\left<c(\partial_{z})^{k-1},\Lied(v) \right>\=\int_{U}c(a'b+kab') dz
\= -\int_{U}b((k-1)a'c+kac') dz, $$
since $d(kabc)=k(a'bc+ab'c+abc')$. The dual of the $k$-Lie derivative is thus
locally given by $\Lied^{\vee}(\tau)=-((k-1)a'c+kac')(dz)^{2}$. Denote by $\Lambda$ the kernel of the map $\Lied^{\vee}$, which
 is given by those $\tau\in H^{0}(X,K_{X}^{1-k}(E))$ that satisfy the differential
equation $(k-1)a'c+kac' = 0$ locally everywhere. The solutions of this equation are simply given
by $c=r a^{(1-k)/k}$, where $r$ is a constant in~$\CC$. Therefore, the kernel $\Lambda$ is at most one-dimensional.
If $a^{1/k}$ is not globally defined, i.e., if~$\xi$ is not the $k$-th power of an abelian differential, then $\Lambda$ is clearly zero.
If $\xi = \omega^k$ for an abelian differential~$\omega$ on $X$, then $E = k \divisor{\omega}$ and $\divisor{a^{(1-k)/k}} = (1-k) \divisor{\omega}$. If $\omega$ is strictly meromorphic, then $r$ has to be zero, for otherwise $\divisor{c}+ E  = \divisor{\omega}$ is not effective, contradicting that $c$ is a section of $\calO_X(E)$, hence in this case $\Lambda$ is also zero. If $\omega$ is holomorphic, the same argument implies that $r$ can be any constant, hence $\Lambda$ is one-dimensional. Summarizing the above discussion, we have shown that $h^{2}(X,\cplxliek)=1$ if $\xi$ is the $k$-th power of a holomorphic abelian differential and is zero otherwise. This implies that the dimension of the tangent space of $\komoduli(m_{1},\ldots,m_{n})$ at $(X,\xi)$ is $2g-1+n$ if $\xi$ is the $k$-th power of a holomorphic abelian differential and is $2g-2+n$ otherwise.
\par
Smoothness of the strata of abelian differentials is known, hence we assume that $\xi$ is not the $k$-th power of an abelian differential. We want that the deformation induced by the cocycle $ \lbrace (D_{\alpha\beta}, \xi'_{\alpha}) \rbrace $ of $\cplxliek$ extends to a higher order deformation.  It is a classical fact  that the deformation of the complex structure induced by $\lbrace D_{\alpha\beta}\rbrace$ extends to higher order if and only if its coboundary in $H^{2}(X,T_{X})$ vanishes (see e.g., \cite{kodaira}). Similarly, we obtain that the deformation given by the cocycle $ \lbrace (D_{\alpha\beta}, \xi'_{\alpha}) \rbrace $ extends if and only if its coboundary in $H^{2}(X,\cplxliek)$ vanishes. We have shown that the second hypercohomology group of $\cplxliek$ is trivial when $\xi$ is not the $k$-th power of an abelian differential. This implies that the deformations of $(X,\xi)$ are unobstructed, and consequently the stratum is smooth at $(X,\xi)$.
\end{proof}
\par
\begin{proof}[Proof of Theorem~\ref{thm:periodCoorkDiff} and Corollary~\ref{cor:period}.]
We start with the case when there are no poles of order $m_{i}\leq -k$, i.e. $P = \emptyset$.
For a singularity $p$ with $\ord_p \xi = m > -k$, one can choose a suitable local coordinate $z$ such that
$\xi = a(z) (dz)^k = z^{m} (dz)^k$ locally at $p$ which corresponds to $z = 0$ (see Proposition~\ref{prop:standard_coordinates}).
By definition of $E_{\rm red}$, a local section $b(z) \tfrac{\partial}{\partial z}$ of $T_{X}(-E_{\rm red})$ at $p$ has the form
$b(z) = \sum_{i=1}^{\infty} b_i z^i$. Then the Lie derivative $\cplxliek$ maps $b$ to
$a'b + kab' = z^{m} \sum_{i=1}^\infty (m + ki ) b_i z^{i-1}$, which can be any element in $z^m\CC[[z]]$.
It follows that $\cplxliek$ is surjective in this case and so the first hypercohomology of $\Lied$
is isomorphic to the first cohomology of $\Lambda_k = {\rm Ker}(\Lied)$. This kernel is a
$\CC$-local system of rank one, generated outside of the zeros of $\xi$ by $a^{-1/k}$ in
the local coordinates used for~\eqref{eq:Liederi}. Since these $a^{-1/k}$ correspond to
the local patches used for the geometric construction of the canonical cover, it follows
that $\Lambda_k$ is trivialized on~$\wh{X}$. More precisely, $\pi^* \Lambda_k =
\wh{\jmath}_{!}\CC$ where $\wh{\jmath}: \wh{X} \setminus \wh{E}_{\rm red} \to
\wh{X}$ is the inclusion defined above. On the other hand, $H^{1}(\whX, \wh{\jmath}_{!}\CC) = H^1(X, \pi_* \wh{\jmath}_{!}\CC)$ by the degeneration
of the Leray spectral sequence,  on which $\tau$ acts equivariantly. Moreover,
$\pi_* \wh{\jmath}_{!}\CC$ splits as a direct sum of its $k$~eigenspaces for the action of~$\tau$,
each of which is a local system of rank one. Since there is a canonical injective map
$\Lambda_k \to \pi_* \pi^* \Lambda_k = \pi_* \wh{\jmath}_{!}\CC$ and since~$\Lambda_k$ is locally
generated by the element $a^{-1/k}$ in the eigenspace where~$\tau$ acts by the $k$-th root
of unity $\zeta$, we obtain that $H^{1}(\whX, \wh{\jmath}_{!}\CC)^{\tau = \zeta} = H^1(X,\Lambda_k)$,
as claimed.
\par
In the general case when there are poles of orders $m_{i}\leq -k$, i.e. $P \neq \emptyset$, recall that we defined the inclusions $\imath: {X} \setminus {P} \to {X}$ and
$\jmath: {X} \setminus {E}_{\rm red} \to {X} \setminus {P}$.  The comparison between singular and sheaf cohomology
and the trivial case of the Leray spectral sequence for maps of relative dimension zero gives the isomorphisms
between the eigenspaces
\bas
H_{1}(\whX \setminus \wh{P}, \wh{Z} ;\CC)_{\tau = \zeta}^\vee
\= H^1(\whX \setminus  \wh{P}, \wh\jmath_{!}\,\CC)_{\tau = \zeta}
\=  H^{1}(\whX, \wh{\imath}_*(\wh\jmath_{!}\,\CC))_{\tau = \zeta},
\eas
and the equivalence of the theorem and the corollary. Since the complex
$$ 0 \to \Lambda_k \to \imath^*(T_X(-\divxired)) \to \imath^*(K_{X}^{k}(-\divxi)) \to 0 $$
on $X \setminus P$ is still exact, we deduce that
\bas
H^1(\whX \setminus  \wh{P}, \wh\jmath_{!}\,\CC)_{\tau = \zeta} \=
H^1(X \setminus P, \Lambda_k)
\= H^1(X, \imath_* \Lambda_k) \= H^1(X, \widetilde{\cplxliek}),
\eas
where $\widetilde{\cplxliek}$ is the complex
\bes
\Bigl(\imath_* \imath^*(T_X(-\divxired)) \to \imath_* \imath^*(K_{X}^{k}(-\divxi)) \Bigr)
\cong \imath_*(\calO_{X \setminus P}) \otimes \Bigl( T_X(-\divxired) \to K_{X}^{k}(-\divxi) \Bigr).
\ees
Since $\imath_*(\calO_{X \setminus P})$ is the sheaf of functions with arbitrary
poles at points in~$P$, there is an obvious inclusion of complexes
$\cplxliek \to \widetilde{\cplxliek}$. To conclude, we have to show that
this inclusion induces an isomorphism on the first cohomology groups. The
cokernel complex consists of the sum over all points $z_j$ with $m_j\leq-k$
of the Lie derivative maps on polar parts
\be \label{eq:LiePolarpart}
\CC((z)) / z \CC[z] \to \CC((z)) / z^{m_j} \CC[z].
\ee
We claim that this map is surjective. In fact, for a singularity~$p$ with
$\ord_p \xi = m \leq -k$ one can choose a suitable local coordinate~$z$
such that  $a(z) = z^m$ if $k\nmid m$, such that $a(z) = \tfrac{r}{z^{k}}$
if $m = -k$, or such that $a(z) = (z^{m/k} + \tfrac{s}{z})^{k}$
if $m < -k$ and $k\mid m$ (see Proposition~\ref{prop:standard_coordinates}). We give the details in the last case, the first two being easier.
Write $m = \ell k$ with $\ell \leq -2$. For any $j \leq 0$ and $b=z^j$
we compute
\bas
a'b + kab' &\= (z^\ell + \tfrac{s}z)^{k-1} z^{j-1} k
((\ell+j) z^\ell +(j-1)\tfrac{s}z) \\
& \= k(\ell+j) z^{m+j-1} + \ \mbox{higher order terms},
\eas
which implies the surjectivity of~\eqref{eq:LiePolarpart}.
Consequently, the first (and of course also the
second cohomology) of~\eqref{eq:LiePolarpart} vanishes and the long exact
sequence of hypercohomology groups implies the claim.
\end{proof}

\subsection{$(1/k)$-translation structure}
\label{subsec:k-trans}

Recall that abelian and quadratic differentials correspond to translation and half-translation surfaces, hence they can be realized as plane polygons with certain edge identification under translation and rotation by $\pi$, respectively. Edges of the polygons as complex vectors correspond to period coordinates, hence in these cases it is geometrically clear that they control deformations of differentials in the respective strata.
\par
In Corollary~\ref{cor:period} we have established period coordinates for the strata of $k$-differentials for all~$k$. We want to describe $k$-differentials by certain generalized translation structure, so that period coordinates can be visible under the corresponding flat geometric presentation.
\par
First, we define a building block that will help us describe the neighborhood of a singularity (under the induced flat metric)
of a $k$-differential. Take $(1/k)$-th of a disk in $\CC$ under the standard Euclidean metric, i.e., a sector bounded by an arc of angle $2\pi / k$. We say that it is a \emph{$(1/k)$-disk}. Take $n$ copies of such $(1/k)$-disks, each with two boundary rays $A_i$ and $B_i$, such that $B_i$ is identified with $A_{i+1}$ by translation for $i=1, \ldots, n-1$. The total angle at the center $p$ of the disks is $2\pi n / k$. Hence for identifying $B_n$ with $A_{1}$, we need rotation by degree $2\pi r/k$, where $r$ is the remainder of $n$ mod $k$. In other words, the last identification is induced by $z \mapsto \zeta^{r} z + c$, where $\zeta = e^{2\pi i/k}$ and $c$ is a constant --- in particular, it is a translation if and only if $k\mid n$. After gluing these disks as above, we say that $p$ is a \emph{cone point of angle $2\pi n /k$}.
\par
Now we define a \emph{$(1/k)$-translation surface} as a closed topological surface $X$ with a finite set of points $\divxi$ satisfying the following conditions:
\begin{itemize}
\item There is an atlas of charts $X\backslash \divxi \to \CC$ whose transition functions are of type
$z \mapsto \zeta^{r} z + c$, where $r\in \ZZ$ and $c$ is a constant.
\item  For each $p\in \divxi$, there is a homeomorphism
of a neighborhood of $p$ to a neighborhood of a cone point of angle $2\pi n /k$, which is an isometry away from $p$.
\end{itemize}

\begin{prop}
\label{prop:k-trans}
There is a one-to-one correspondence between $(1/k)$-translation surfaces and $k$-differentials whose pole orders are at most $k-1$. In particular, a singularity of a $k$-differential with order $m$ corresponds to a cone point of angle $2\pi (k+m)/k$.
\end{prop}

\begin{proof}
At a singularity $p$ of a $k$-differential $\xi$, we choose a suitable local coordinate $z$ such that
$\xi = z^m (dz)^{k}$ (see Proposition~\ref{prop:standard_coordinates} for the general statement and proof, noting that here we are dealing with poles of order at most $k-1$). Rewrite it formally as $(z^{m/k} dz)^k$. Then a flat geometric neighborhood of $p$ can be identified with gluing $(k+m)$ copies of a $(1/k)$-disk as above, hence the claim follows.
\end{proof}

As in the cases of abelian and quadratic differentials, one can equivalently describe $k$-differentials in terms of cutting and pasting polygons (see \cite[Definitions 1.1, 1.5, and 1.7]{wrightSurvey}). In this sense, a $(1/k)$-translation surface is a collection of polygons in $\CC$ with edge identifications as follows:
\begin{itemize}
\item All the edges are grouped in pairs such that the edges in the same pair are identified by translation and rotation by degree $2\pi n/k$ for some $n\in \ZZ$.
\item For the edges in the same pair, the polygon interior nearby one edge does not map to the polygon interior nearby the other edge.
\end{itemize}

Two such collections of polygons define the same $(1/k)$-translation surface if and only if one can be refined and re-glued to form the other by translations and rotations of degree $2\pi n/k$ for some $n\in \ZZ$.
\par
For meromorphic abelian differentials, the flat geometric neighborhood of a pole can be described by half-infinite cylinders or by gluing (broken) half-planes along their boundary via translation, as introduced by Boissy~\cite{boissymero}. Such construction can be similarly carried out for $k$-differentials with arbitrary pole orders. The building blocks for a pole of order $\geq k$
are either half-infinite cylinders or (broken) $(1/2k)$-planes, where the edge identification is given by translation and rotation of degree $2\pi n /k$ for $n\in \ZZ$.
\par
Finally we provide an example to illustrate the flat geometric meaning of period coordinates in Corollary~\ref{cor:period}.
\begin{exa}
Take a primitive $\xi \in \komoduli((2g-2)k)$. Let $z$ be
the unique zero of $\xi$. Since $k\mid (2g-2)k$,
the canonical cover $\pi: \whX\to X$ is unramified. Let $\tau$ be the deck transformation of $\whX$ corresponding to $\zeta$, and let $\wh z_1, \ldots, \wh z_k$ be the preimages of $z$ such that $\tau(\wh z_i) = \wh z_{i+1}$. By the Riemann-Hurwitz formula,
$$ 2\whg - 2 \= k (2g-2), $$
hence we have
$$ \dim H_1(\whX, \CC)/\pi^{*} H_1(X, \CC) \= 2\whg - 2g \= (k-1)(2g-2).$$
There is a basis $\wh\gamma_1, \ldots, \wh\gamma_{2g-2}$ of absolute homology classes that belong to $V_\zeta$. For the relative part, take a path $\alpha_i$ that joins $\wh z_i$ to $\wh z_{i+1}$ for $i = 1, \ldots, k$ such that
$\tau(\alpha_i) = \alpha_{i+1}$ in $H_1(\whX, \wh z_1, \ldots,\wh z_k; \CC)$. Let
$$\gamma_{2g-1} \:= \sum_{i=1}^k \zeta^{-i}\alpha_i. $$
Then
\begin{equation}\label{eq:equivhomo}
\tau^{*}\gamma_{2g-1} \= \sum_{i=1}^k \zeta^{-i} \alpha_{i-1} \= \zeta \gamma_{2g-1},
\end{equation}
hence $\gamma_{2g-1}\in V_\zeta$ and lies in the relative part. Therefore, $\wh\gamma_1, \ldots, \wh\gamma_{2g-2}, \gamma_{2g-1}$
together form a basis of $V_\zeta$.
\par
The case of  $g=2$ and $k=3$ is pictured in Figure~\ref{fig:triple}. The bottom surface represents a primitive cubic differential $\xi$ in $\Omega^3\moduli[2](6)$, where the edges with the same labels are identified by translation (for $a$ and~$b$) and rotation of degree $2\pi / 3$ (for $c$ and $d$). In particular, the three vertices of the inner triangle are identified as the same point which forms the unique zero
of $\xi$. In the top three surfaces edges with the same labels are identified by translation, hence after gluing we obtain an abelian differential $\omega$, and the resulting surface admits a canonical triple cover of the bottom which maps $a_0, a_1, a_2$ to $a$, etc.
Note that in this case $\omega$ lies in the stratum $\omoduli[4](2, 2, 2)$. The absolute part of the equivariant homology consists of weighted sums of the preimages of $a$ and of~$b$ with weight pattern as in Equation~\eqref{eq:equivhomo}. The relative part is given by Equation~\eqref{eq:equivhomo}, where the $\alpha_{i}$ are the preimages of $c$.

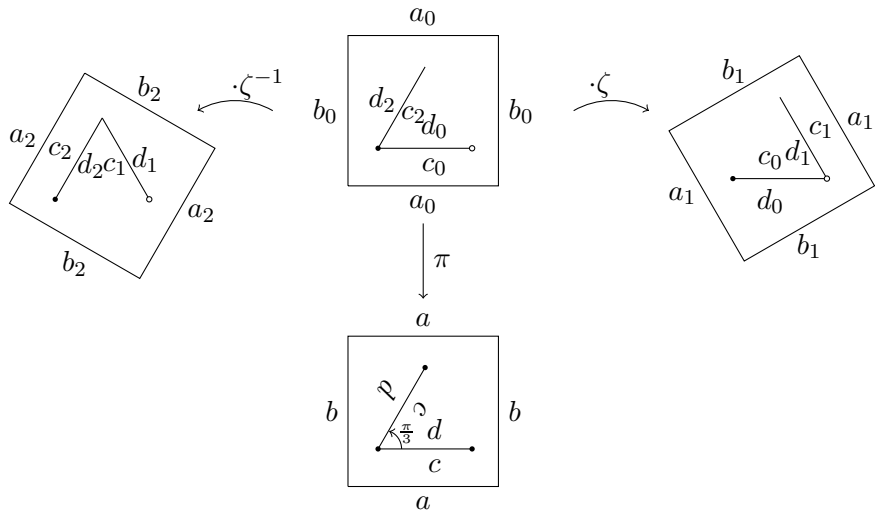
\begin{figure}[htb]
\centering
\begin{tikzpicture}

 \begin{scope}[xshift=-4.5cm,yshift=1.5cm,rotate=-120]

\draw[] (0,0)  -- coordinate[pos=.5](a1) ++(2,0) -- coordinate[pos=.5](b1)  ++(0,2)-- coordinate[pos=.5](a2) ++(-2,0)-- coordinate[pos=.5](b2) ++(0,-2);
 \draw (.4,.5) coordinate (P) -- coordinate[pos=.4](c1)coordinate[pos=.25](d) ++(1.25,0) coordinate(Q);
  \draw (P) -- coordinate[pos=.6](c2) ++(60:1.25) coordinate(R);

 \node[left] at (a1) {$a_{2}$};
 \node[right] at (a2) {$a_{2}$};
 \node[below] at (b1) {$b_{2}$};
 \node[above] at (b2) {$b_{2}$};
 \node[xshift=.1cm,yshift=-.2cm] at (c1) {$d_{2}$};
 \node[left] at (c1) {$c_{2}$};
 \node[xshift=-.2cm] at (c2) {$c_{1}$};
 \node[xshift=.2cm,yshift=.1cm] at (c2) {$d_{1}$};

  \fill (Q) circle (1pt);
  \fill[white] (R) circle (1pt);    \draw (R) circle (1pt);

 \end{scope}

 \begin{scope}[xshift=-1cm]
\draw[] (0,0)  -- coordinate[pos=.5](a1) ++(2,0) -- coordinate[pos=.5](b1)  ++(0,2)-- coordinate[pos=.5](a2) ++(-2,0)-- coordinate[pos=.5](b2) ++(0,-2);
 \draw (.4,.5) coordinate (P) -- coordinate[pos=.6](c1)coordinate[pos=.25](d) ++(1.25,0) coordinate(Q);
  \draw (P) -- coordinate[pos=.6](c2) ++(60:1.25) coordinate(R);

 \node[below] at (a1) {$a_{0}$};
 \node[above] at (a2) {$a_{0}$};
 \node[right] at (b1) {$b_{0}$};
 \node[left] at (b2) {$b_{0}$};
 \node[below] at (c1) {$c_{0}$};
 \node[above] at (c1) {$d_{0}$};
 \node[xshift=.1cm,yshift=-.2cm] at (c2) {$c_{2}$};
 \node[left] at (c2) {$d_{2}$};

  \fill[white] (Q) circle (1pt);    \draw (Q) circle (1pt);
  \fill (P) circle (1pt);

  \draw[->] (-1,1) arc  (60:120:1);\node[] at (-1.2,1.35) {$\cdot\zeta^{-1}$};
  \draw[->] (3,1) arc  (120:60:1);\node[] at (3.35,1.35) {$\cdot\zeta$};

 \draw[->] (1,-.5) -- (1,-1.5);\node[right] at (1,-1) {$\pi$};

 \begin{scope}[yshift=-4cm]
\draw[] (0,0)  -- coordinate[pos=.5](a1) ++(2,0) -- coordinate[pos=.5](b1)  ++(0,2)-- coordinate[pos=.5](a2) ++(-2,0)-- coordinate[pos=.5](b2) ++(0,-2);
 \draw (.4,.5) coordinate (P) -- coordinate[pos=.6](c1)coordinate[pos=.25](d) ++(1.25,0) coordinate(Q);
  \draw (P) -- coordinate[pos=.6](c2) ++(60:1.25) coordinate(R);

 \draw[->] (d) arc  (0:70:.25);
\node[] at (.78,.7) {\tiny{$\frac{\pi}{3}$}};

 \node[below] at (a1) {$a$};
 \node[above] at (a2) {$a$};
 \node[right] at (b1) {$b$};
 \node[left] at (b2) {$b$};
 \node[below] at (c1) {$c$};
 \node[above] at (c1) {$d$};
 \node[below,rotate=-120] at (c2) {$d$};
 \node[above,rotate=-120] at (c2) {$c$};

  \fill (P) circle (1pt);
  \fill (Q) circle (1pt);
  \fill (R) circle (1pt);

       \end{scope}
 \end{scope}

 \begin{scope}[xshift=6cm,rotate=120]

\draw[] (0,0)  -- coordinate[pos=.5](a1) ++(2,0) -- coordinate[pos=.5](b1)  ++(0,2)-- coordinate[pos=.5](a2) ++(-2,0)-- coordinate[pos=.5](b2) ++(0,-2);
 \draw (.4,.5) coordinate (P) -- coordinate[pos=.6](c1)coordinate[pos=.25](d) ++(1.25,0) coordinate(Q);
  \draw (P) -- coordinate[pos=.6](c2) ++(60:1.25) coordinate(R);

 \node[right] at (a1) {$a_{1}$};
 \node[left] at (a2) {$a_{1}$};
 \node[above] at (b1) {$b_{1}$};
 \node[below] at (b2) {$b_{1}$};
 \node[below] at (c1) {$d_{1}$};
 \node[right] at (c1) {$c_{1}$};
 \node[below] at (c2) {$d_{0}$};
 \node[above] at (c2) {$c_{0}$};

   \fill (R) circle (1pt);
  \fill[white] (P) circle (1pt);    \draw (P) circle (1pt);

 \end{scope}
 \end{tikzpicture}
  \caption{Canonical cover of a primitive cubic differential in $\Omega^3\moduli[2](6)$}
\label{fig:triple}
\end{figure}
\end{exa}

In Section~\ref{sec:examples} we will use the $(1/k)$-translation surface viewpoint to present a number of examples highlighting various aspects of the incidence variety compactification that will be discussed soon.


\section{Stable $k$-differentials and $k$-residues} \label{sec:stablek}

In this section we define in detail the ambient space that we will use
for the incidence variety compactification. For this purpose we need a local
standard form for $k$-differentials, and the notion of $k$-residues.
We will also recall the definition of level graphs and twisted differentials as introduced in~\cite{plumb}.

\subsection{Standard coordinates for $k$-differentials, and $k$-residues}
\label{subsec:k-residue}
Contrary to the case of abelian differentials, not every pole of a $k$-differential $\xi$ has a local invariant like the residue, independent of the choice of a local coordinate, but only poles of order divisible by~$k$ have such an invariant. We make it precise in the following proposition.
\begin{prop} \label{prop:standard_coordinates}
Given a $k$-differential $\xi$ on a sufficiently small disk~$V$
around $0$, denote by $m$ the order $\ord_0\xi$.
Then there exists a conformal map $\phi\colon(\Delta_R, 0) \to (V, 0)$ defined
on a disk of sufficiently small radius~$R$, and a number $r \in \CC$ such that
\begin{equation}\label{eq:standard_coordinates}
    \phi^*(\xi) \=
    \begin{cases}
      z^m\, (dz)^{k} &\text{if $m> -k$ or  $k\nmid m$,}\\
      \frac{r}{z^{k}}(dz)^{k} &\text{if $m = -k$,}\\
        \left(z^{m/k} + \frac{s}{z}\right)^{k}(dz)^{k} &\text{if $m < -k$ and $k\mid m$}
    \end{cases}
\end{equation}
where in the last case $s \in \CC$ and $s^k = r$. The germ of $\phi$ is unique up to
multiplication by a $(m+k)$-th root of unity when $m> -k$ or $k\nmid m$, and up to multiplication by
a non-zero constant if $m=-k$.
\end{prop}
We call the coordinate~$z$ provided in the proposition
the {\em standard coordinate} for $\xi$, call the number~$r$ the {\em $k$-residue} of~$\xi$, and
denote it by $r=\Resk_0\xi$. For convenience we also set $\Resk_0 \xi = 0$ if
$m> -k$ or $k\nmid m$. We remark that while for $m=-k$ the residue $\Resk_0\xi$ is simply the order $-k$ term of the Laurent series of $\xi$ at $0$, for $m\ne -k$, the $k$-residue does not admit such an easy description.
\par
This proposition for quadratic differentials (i.e. $k = 2$) was stated and proved in~\cite[Paragraph~6]{strebel}. The same argument works for general $k$ as well. For the reader's convenience, we include a proof as follows.
\begin{proof} Suppose in a
local coordinate $w$ around~$0$ the $k$-differential $\xi$ is given by
 \[\xi=w^{m}\left( a_{0}+a_{1}w+\cdots \right) (dw)^{k},\]
 with $a_{0}\neq 0$. Setting $w=\varphi(u)=u^{k}$
we obtain that
 \[\varphi^{\ast}(\xi)=u^{mk}\left( a_{0}+a_{1}u^{k}+\cdots \right) (ku^{k-1}du)^{k}.\]
Note that the $k$-th root of the power series  $a_{0}+a_{1}u^{k}+\cdots$ is
again a power series $b_{0}+b_{1}u^{k}+\cdots$ in $u^k$ with the property that
 \[\varphi^{\ast}(\xi)=  \left(u^{m+k-1}\left( b_{0}+b_{1}u^{k}+\cdots\right) k du\right)^{k}.\]
Let $\psi$ be a coordinate change that puts the abelian differential inside
the parenthesis on the right hand side into the standard form. In the first case $m > -k$ or
$m+k-1 \not \equiv -1 \pmod{k}$, this abelian differential has no residue, and hence
$$\psi^\ast \varphi^{\ast}(\xi) \=   (v^{m+k-1} dv)^k \= \frac{1}{k^k}
\, v^{mk} (d(v^k))^k\,.$$
Since this differential depends on $v^k$ only, setting $z=v^k$ provides the desired expression of the standard
coordinate. Moreover, for $v$ to satisfy that
$u^{m+k-1}(b_0 + b_1 u^k + \cdots ) k du = v^{m+k-1} dv$,
it suffices to have $u^{m+k}(c_0 + c_1 u^k + \cdots ) = v^{m+k}$ by integrating both sides, where $c_0\neq 0$. Writing $v = u  y$, the condition reduces to
$c_0 + c_1 u^k + \cdots  = y^{m+k}$, hence $y$ can be expressed in terms of $u^k$, which implies that $z = v^k = u^k y^k$ is a well-defined coordinate change from the original coordinate $w = u^k$.
\par
If $m=-k$, the existence of standard coordinates for abelian differentials implies the existence of $\psi$ such that
$$\psi^\ast \varphi^{\ast}(\xi) \=  k^k (s v^{-1} dv)^k \= r (v^{-k} d(v^k))^k\,.$$
Again this is the pullback under $v\mapsto z=v^k$ of a $k$-differential which is of the form $\tfrac{r}{z^k}(dz)^k$. The last case of $m<-k$ and $k\mid m$ is similar, except a more delicate point that now the abelian differential written in standard coordinate may have a residue.
\end{proof}
\par
From the flat geometric viewpoint of Section~\ref{subsec:k-trans}, the $k$-residue corresponds to a half-infinite cylinder for a pole of order $-k$ and to a half-infinite slit for poles of higher order (see the right of Figure~\ref{fig:triplepres}). We stress that the width of the slit or of the cylinder represents some $k$-th root of the $k$-residue.

\subsection{Stable $k$-differentials}
\label{subsec:stablek}

Let $f: \calC \to \barmoduli[g,n]$ be the universal curve over the Deligne-Mumford moduli space of $n$-pointed stable genus $g$ curves,
and let $\omega_f$ be the relative dualizing sheaf. Consider the space $\kobarmodulin := f_{*}(\omega_f^{k})$ as a vector bundle over $\barmoduli[g,n]$. This is to say, its fibers are spaces $H^0(X, \omega_X^k)$ parameterizing stable $k$-differentials $\xi$ over pointed stable curves $(X,z_1,\ldots,z_n)$, where $\omega_X$ is the dualizing sheaf of $X$. We call $(X,\xi,z_1,\ldots,z_n)\in\kobarmodulin$ {\em pointed stable (holomorphic) $k$-differentials}. In particular, it means that $X$ is at worst nodal, $z_1,\ldots,z_n$ are distinct smooth points of $X$, and $\xi$ is a collection of $k$-differentials $\xi_v$ on irreducible components $X_v$ of $X$, which may have poles of order up to $k$ at the nodes of $X$. Furthermore, $\xi$ is required to satisfy the matching residue condition, namely, at any node $q\in X$ identifying $q_1\in X_{v_1}$ with $q_2\in X_{v_2}$ the equality $\Resk_{q_1}\xi_{v_1}=(-1)^k\Resk_{q_2}\xi_{v_2}$ holds (as discussed in Section~\ref{subsec:k-residue}, for
$k$-differentials with poles of order at most $k$ the $k$-residue is equal to the coefficient of $z^{-k}(dz)^k$ in the Laurent expansion, and this equality for $k$-residues follows from  the matching residue condition for $\omega_X$ at $q$).
\par
Similarly to the case of meromorphic abelian differentials, for meromorphic $k$-differentials we have to twist the ambient space by the polar part of the differentials. We thus write $\mu=(m_1,\ldots,m_r,m_{r+1},\ldots,m_{r+s},m_{r+s+1},\ldots,m_n)$ with $m_1,\ldots,m_r>0$, $m_{r+1}=\cdots=m_{r+s}=0$, and $m_{r+s+1},\ldots,m_n<0$.
Denote by $\tilde{\mu} = (m_{r+s+1},\ldots,m_n)$ the polar part of~$\mu$. Now we define
$\kobarmodulin(\tilde{\mu}):= f_{*} ( \omega_f^{ k}(-\sum_{i=r+s+1}^n m_i Z_i))$, where the $Z_i$ are the sections of the universal curve corresponding to the marked points. We call points
$(X,\xi,z_1,\ldots,z_n)\in\kobarmodulin(\tilde{\mu})$  {\em pointed stable $k$-differentials}, which includes the holomorphic case by setting $\tilde{\mu} = \emptyset$.
 \par
Recall that the stratum $\komoduli[g](\mu)$ parameterizes $(X, \xi, z_1, \ldots, z_n)$ where $X$ is a smooth genus~$g$ curve, $z_1, \ldots, z_n$ are distinct marked points, and $\xi$ is a $k$-differential with zeros and poles at $z_i$ such that $\ord_{z_i}\xi = m_i$. It is clear that $\komoduli[g](\mu)$ is contained in $\kobarmodulin(\tilde{\mu})$. Note that there is a $\CC^*$-action on $\kobarmodulin(\tilde{\mu})$ by multiplication, which preserves the stratification of $\kobarmodulin(\tilde{\mu})$. We denote the corresponding quotients by adding a letter~$\PP$. Now we define the {\em incidence variety compactification} $\kivc{\mu}$ of the stratum $\komoduli[g](\mu)$ to be the closure of the stratum in $\PP\kobarmodulin(\tilde\mu)$. We remark that our notation implicitly means that the $n$ marked points are ordered. For the unordered version of the compactification, one can simply take the quotient by the respective symmetric group action.
\par
As explained in \cite{plumb}, the incidence variety compactification keeps both the scales of the limit differentials and the limit positions of zeros and poles, as $k$-differentials degenerate in a given stratum. Moreover, it determines the strata compactification in both the Deligne-Mumford space and in the Hodge bundle
by forgetting the differentials and the markings respectively. As can be seen by looking at $k$-th powers of abelian differentials, for which this is discussed in~\cite{plumb}, the incidence variety compactification records in general strictly more information than taking the strata compactification in either of the latter two spaces.

\subsection{Level graphs and twisted differentials}
\label{subset:level}

In order to describe the incidence variety compactification of the strata of $k$-differentials, we need to introduce twisted $k$-differentials. The terminology we use here is closely related to that of~\cite{plumb}, with the extra complications
coming from the particularities of $k$-residues.
\par
\begin{df}\label{df:twdk}
Let $\mu = (m_1,\ldots,m_n)$ be a tuple of integers with $\sum_{i=1}^n m_i = k(2g-2)$.
A {\em twisted $k$-differential $\eta$ of type $\mu$}
on a stable $n$-pointed curve $(X,z_1,\ldots,z_n)$
is a collection of (possibly meromorphic)
$k$-differentials $\eta_v$ on the irreducible components~$X_v$ of~$X$
such that no~$\eta_v$ is identically zero and the following conditions hold:
\begin{itemize}
\item[(0)] {\bf (Vanishing as prescribed)} Each $k$-differential $\eta_v$ is holomorphic and non-zero outside the nodes and marked points of $X_v$. Moreover, if a marked point $z_i$ lies on~$X_v$, then $\ord_{z_i} \eta_v=m_i$.
\item[(1)] {\bf (Matching orders)} For any node of $X$ that identifies $q_1 \in X_{v_1}$ with $q_2 \in X_{v_2}$,
$$\ord_{q_1} \eta_{v_1}+\ord_{q_2} \eta_{v_2} \= -2k. $$
\item[(2)] {\bf (Matching residues at poles of order $k$)} If at a node of $X$
that identifies $q_1 \in X_{v_1}$ with $q_2 \in X_{v_2}$ the condition $\ord_{q_1}\eta_{v_1}=
\ord_{q_2} \eta_{v_2}=-k$ holds, then
$$\Resk_{q_1}\eta_{v_1} \= (-1)^k\Resk_{q_2}\eta_{v_2}.$$
\end{itemize}
\end{df}
\par
We next recall the notion of level graphs (see also \cite[Section~1.5]{plumb}). Let $\Gamma$ be the dual graph of a nodal curve $X$.
A {\em full order} on the irreducible components of $X$ is a relation~$\succcurlyeq$ on
the set $V$ of vertices of $\Gamma$ that is reflexive, transitive,
and such that for any $v_1, v_2 \in V$ at least one of the statements
$v_1 \succcurlyeq v_2$ or $v_2 \succcurlyeq v_1$ holds. We remark that
equality (denoted by $\asymp$) is permitted in our definition of a full order.
Any map $\ell:V\to\RR$ assigning real numbers to vertices of $\Gamma$ defines a full
order on $\Gamma$ by setting $v_1 \succcurlyeq v_2$ if and only if $\ell(v_1) \geq \ell(v_2)$.
We call a dual graph $\Gamma$ equipped with a full order on its vertices a {\em level graph},
and denote it by~$\overline{\Gamma}$. Given a level function $\ell$ and
a real number $L$, we denote by $\overline\Gamma_{>L}$ (resp.~$\overline\Gamma_{\ge L}$, $\overline\Gamma_{=L}$, etc)
the subgraph of~$\Gamma$ consisting of all vertices such that $\ell(v)>L$ (resp.~$\ell(v)\ge L$, $\ell(v)=L$, etc)
and all the edges connecting them.
\par
In a level graph, an edge is called {\em horizontal} if it joins two vertices
at the same level, and called {\em vertical} otherwise. We also use the same names
for the corresponding nodes of the curve. For a node $q$ joining two components $X_{v_1}$ and $X_{v_2}$ with
$v_1 \succcurlyeq v_2$, we denote by $q^+$ and $q^-$ the preimages of $q$ in $X_{v_1}$ and $X_{v_2}$, respectively. We also denote by $v^+(q) =v_1$ and $v^-(q) = v_2$ in this case.
We will draw level graphs in such a way that the level function is given by projection to the
vertical axis, so that the top level corresponds to the top of the pictured graph.
\par
Along with the compatibility conditions (3) and (4-ab) in Definition~\ref{def:twistedAbType} for twisted abelian differentials, we recall the main result of~\cite{plumb}.
\begin{thm}[\cite{plumb}]
\label{thm:main:plumb1}
A pointed stable {\em abelian} differential $(X,\omega,z_1,\ldots,z_n)$ is contained in the
incidence variety compactification of a
stratum $\proj\omoduli(\mu)$ if and only if the following conditions hold:
\begin{itemize}
\item[(i)] There exists a level graph $\overline\Gamma$ such that its maxima are the irreducible components $X_v$ of $X$ on which~$\omega$ is not identically zero.
\item[(ii)] There exists a twisted abelian differential~$\eta$ of type $\mu$ on~$X$, compatible with $\overline\Gamma$.
\item[(iii)] On every irreducible component $X_v$ where $\omega$ is not identically zero, $\eta_v = \omega|_{X_v}$.
\end{itemize}
\end{thm}


\section{Normalized covers and the global residue condition}\label{sec:lift}
In this section we prove Theorem~\ref{thm:kmain}, where the compatibility condition uses the \whgrc~($\widehat{4}$), by reducing the problem to the case of abelian differentials (Theorem~\ref{thm:main:plumb1}). This requires extending the notion of
canonical covers, defined in Section~\ref{sec:DefOfMeroDiff} for smooth curves,
to the case of nodal curves. More precisely, given a nodal curve $X$ with a twisted $k$-differential $\eta$,
we want to construct a (possibly disconnected) admissible cover $\pi: \wh X \to X$ such that $\pi^{*}\eta = \wh \omega^k$, where $\wh \omega$ is a twisted abelian differential on $\wh X$.
Here the admissibility means that every node of~$\wh X$ maps to a node of $X$, every preimage of a node of $X$ is a node of $\wh X$, and the ramification orders of $\pi$ at the two branches of any node of~$\wh X$ are equal (see~\cite[Paragraph~3.G]{hamobook}). In addition, we want $\wh X$ to carry a deck transformation $\tau$, cyclic of order $k$ as in the case of smooth canonical covers.
Note that if $q$ is a node of~$X$ identifying $q_1\in X_{v_1}$ with $q_2\in X_{v_2}$ and if $\eta$ is a twisted
$k$-differential on~$X$, then~$q_i$ has
$\gcd(k,\ord_{q_i}\eta_i)$ distinct preimages in the canonical cover of~$X_{v_i}$.
Thus by the matching orders condition for twisted $k$-differentials (condition (1) in Definition~\ref{df:twdk}), $q_1$ and $q_2$ have the same
number of preimages in the canonical covers $\whX_{v_1}$ and~$\whX_{v_2}$. Hence to construct $\wh X$, we can pair preimages of $q_1$ with preimages of $q_2$ and identify the two points in each pair to form a node. In general there are many choices, but since $\wh X$ is required to have a cyclic deck transformation $\tau$, we only need to identify one pair of the preimages of $q_1$ and $q_2$, and the remaining identifications are then determined from the induced actions of $\tau$ on $\whX_{v_1}$ and~$\whX_{v_2}$.
\par
To give the precise definition of the covers we want, fix a type $\mu=(m_1,\ldots,m_n)$ of $k$-differentials. We let $\wh m_i := \tfrac{k+m_{i}}{\gcd(k,m_{i})}-1$
and let
$$
  \whmu :=\Bigl(\underbrace{\wh m_1, \ldots, \wh m_1}_{\gcd(k,m_{1})},\,\underbrace{\wh m_2,
  \ldots, \wh m_2}_{\gcd(k,m_{2})} ,\ldots,\,
  \underbrace{\wh m_n, \ldots, \wh m_n}_{\gcd(k,m_{n})} \Bigr).
$$
Given a cover $\pi: \whX \to X$ of a pointed stable curve $X$, we mark on $\whX$ all
preimages of all the marked points in $X$. The order of the marked preimages in $\whX$ will not matter,
except that we list the preimages of the first marked point first, etc.
\par
Given a stable curve~$X$ together with a twisted $k$-differential~$\eta$,
consider a triple $$(\pi: \whX \to X, \tau, \whomega)$$
consisting of an admissible cover $\pi$ of degree $k$, its deck transformation~$\tau$ of order $k$, and a twisted abelian differential $\whomega$ on $\whX$ such that
$\tau^* \whomega = \zeta \whomega$ and $\whomega^k = \pi^* \eta$ on~$\whX$, where $\zeta$ is the chosen
primitive $k$-th root of unity. We will still call such triple a {\em normalized cover} (of a nodal curve), and note that its
restriction to any irreducible component~$X_v$ of~$X$ is the normalized canonical
cover as defined in Section~\ref{sec:DefOfMeroDiff}.
\par
We will often abbreviate notation and say that $(\whX, \wh{\omega})$ or $\whX$ is a normalized cover of~$X$.
Note that if such a normalized cover exists, then
$\whmu$ is the only type $\whomega$ can possibly
have, since the preimage of a marked point~$z$ with
$\ord_{z} \eta = m$ consists of $\gcd(k,m)$ distinct points at which $\whomega$ has equal order. We now show that  normalized covers always exist.
\par
\begin{lm}\label{lemma:existenceLift}
For any twisted $k$-differential $\eta$ on a pointed stable curve $X$, there exists a
normalized cover  $(\whX, \tau, \whomega)$ such that $\whomega$ is a twisted
abelian differential of type $\wh\mu$.
\end{lm}
\par
\begin{proof}
We first take normalized covers $\whX_v\to X_v$ of all irreducible components of $X$, with meromorphic abelian differentials $\whomega_v$ on $\whX_v$. To construct $\whX$, it remains to identify the preimages of two branches of each node of $X$ in such a way that $\whomega$ satisfies the conditions for a twisted abelian differential on $\whX$. By the construction of normalized covers for smooth curves, $\whomega$ has prescribed vanishing, that is, it satisfies condition (0) in Definition~\ref{df:twdk} for twisted abelian differentials (i.e., for $k=1$). To verify the matching orders condition (1), note that if $\ord_{q^{\pm}} \eta = -k\pm m$ at a node~$q$, then for any preimage $\wh{q}_i$ of $q$
$$
  \ord_{\whq^{\pm}_i} \whomega \= \frac{k + (-k \pm m) }{\gcd (k,m)} - 1 \= \frac{\pm m}{\gcd(k,m)}-1,
$$
hence $\ord_{\whq^{+}_i} \whomega + \ord_{\whq^{-}_i} \whomega = -2$, and consequently condition (1) holds, independent of how the lifted branches of the nodes are identified.
\par
If $q$ is a vertical node of $X$, then there is no matching residue condition imposed at its preimages, and the only restriction is that the curve would have an order $k$ automorphism $\tau$. Hence one can choose any one preimage of $q^+$ and identify it with any one preimage of $q^-$ to form a node, and then the identifications of the rest of the preimages of $q^{\pm}$ are determined by the group action.
\par
Finally, if $q_1\in X_{1}$ and $q_2 \in X_2$ are identified to form a horizontal node $q \in X$, i.e., if $\ord_{q_1}\eta_1=\ord_{q_2}\eta_2=-k$, then the matching residue condition (2) for twisted $k$-differentials requires
$\Resk_{q_1}\eta_1=(-1)^k\Resk_{q_2}\eta_2$. Note that in this case these $k$-residues are non-zero, as they are given by the coefficients of $z^{-k}(dz)^k$ in the local expression of $\eta$. The  canonical covers of $X_i$ are then unramified over $q_i$, and the residues of $\whomega$ at the preimages of $q_i$ are the $k$-th roots of $\Resk_{q_i}\eta_i$, which differ from each other by multiplication by $k$-th roots of unity. Hence for each preimage of $q_1$, there exists a unique preimage of $q_2$ where $\whomega$ has opposite residue. Identifying these two preimages with opposite residues ensures that $\whomega$ satisfies the matching residue condition (2) for twisted abelian differentials (and is the unique choice for such an identification).
\end{proof}
\par
Recall that every full order on a dual graph is induced by some level function. Let $\overline\Gamma$ be a full order on the dual graph $\Gamma$ of $X$, which is given by the level function $\ell:V(\Gamma)\to\RR$. For any admissible cover $\pi:\wh X\to X$, we define
the {\em lifted full order} (or {\em lifted level graph}) $\wh{\overline\Gamma}$ on the dual graph $\wh{\Gamma}$ of
$\wh{X}$ to be the full order given by the level function $\ell\circ\pi$ on $V(\wh \Gamma)$.
It is easy to see that the lifted full order only depends on $\overline\Gamma$ and not on the choice of the level function.
\par
\begin{lm}\label{lemma:liftcompatibleorder}
Let $\eta$ be a twisted $k$-differential on a pointed stable curve $X$ satisfying the
partial order condition $(3)$ for a level graph $\overline\Gamma$. Then
for every normalized cover $(\pi: \whX \to X,\tau, \whomega)$, the twisted abelian differential
$\whomega$ also satisfies the partial order condition $(3)$ for the
lifted level graph~$\wh{\overline\Gamma}$.
\par
Moreover for any level $L$, the restriction of the normalized cover over $X_{=L}$ is uniquely determined by $\eta$, and the entire cover is further determined by choosing, for each vertical node $q$ of~$\overline\Gamma$, one $\pi$-preimage of~$q^+$ and one $\pi$-preimage of $q^-$ that are identified.
\end{lm}
\par
\begin{proof}
 For any node $q$ of $X$, we have $\ord_{q^\pm}\eta \geq -k$ if and only if $\ord_{\wh q_i^\pm}\whomega \geq -1$ for any
 preimage $q_i$ of $q$. It follows that $\whomega$ satisfies the partial order condition (3) for $\wh{\overline\Gamma}$.
The uniqueness of identifying preimages of a horizontal node on level $L$ and the way of identifying preimages of a vertical node were shown in the proof of Lemma~\ref{lemma:existenceLift} already.
\end{proof}
\par
Before proving Theorem~\ref{thm:kmain}, we remark that condition ($\widehat{4}$) depends on the
choice of a normalized cover, as shown by the following example.
\par
\begin{exa}[The \whgrc~depends on the normalized cover]
\label{ex:GRCdependLift}
Let $X$ be a stable curve consisting of three irreducible components $X_{1}$, $X_{2}$, and $X_{3}$ such that $X_{1}$ meets $X_{2}$ at two nodes and $X_{2}$ meets $X_{3}$ at one node (see Figure~\ref{cap:GRC}). Let $\eta$ be a twisted quadratic differential (i.e., $k=2$) such that~$\eta_{1}$ and~$\eta_{2}$ are squares of abelian differentials and $\eta_{3}$ is a primitive quadratic differential. Moreover, suppose $\eta_{2}$ is the square of a meromorphic abelian differential that has poles at the nodes of $X_{1}$ without residue. Let us consider the two normalized covers $\whX^{(1)}$ and $\whX^{(2)}$ with lifted level graphs shown in Figure~\ref{cap:GRC}.
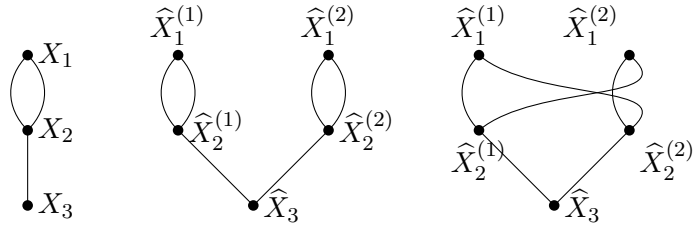
\begin{figure}[ht]
\centering
\begin{tikzpicture}[scale=1]
\fill (0,0) coordinate (x1) circle (2pt); \node [right] at (x1) {$X_{1}$};
\fill (0,-1) coordinate (x2) circle (2pt); \node [right] at (x2) {$X_{2}$};
\fill (0,-2) coordinate (x3) circle (2pt); \node [right] at (x3) {$X_{3}$};
 \draw (x1) ..controls (.3,-.3) and (.3,-.7)  ..  (x2);
  \draw (x1) ..controls (-.3,-.3) and (-.3,-.7)   ..  (x2);
 \draw (x2) -- (x3);
  \begin{scope}[xshift=3cm]
  \fill (-1,0) coordinate (x11) circle (2pt); \node [above] at (x11) {$\whX_{1}^{(1)}$};
\fill (-1,-1) coordinate (x21) circle (2pt); \node [right] at (x21) {$\whX_{2}^{(1)}$};
 \fill (1,0) coordinate (x12) circle (2pt); \node [above] at (x12) {$\whX_{1}^{(2)}$};
\fill (1,-1) coordinate (x22) circle (2pt); \node [right] at (x22) {$\whX_{2}^{(2)}$};
\fill (0,-2) coordinate (x3) circle (2pt); \node [right] at (x3) {$\whX_{3}$};
 \draw (x11) ..controls (-.7,-.3) and (-.7,-.7)  ..  (x21);
  \draw (x11) ..controls (-1.3,-.3) and (-1.3,-.7)   ..  (x21);
  \draw (x12) ..controls (.7,-.3) and (.7,-.7)  ..  (x22);
  \draw (x12) ..controls (1.3,-.3) and (1.3,-.7)   ..  (x22);
 \draw (x21) -- (x3);
  \draw (x22) -- (x3);
 \end{scope}
 \begin{scope}[xshift=7cm]
 \fill (-1,0) coordinate (x11) circle (2pt); \node [above] at (x11) {$\whX_{1}^{(1)}$};
\fill (-1,-1) coordinate (x21) circle (2pt); \node [below] at (x21) {$\whX_{2}^{(1)}$};
 \fill (1,0) coordinate (x12) circle (2pt); \node [above left] at (x12) {$\whX_{1}^{(2)}$};
\fill (1,-1) coordinate (x22) circle (2pt); \node [below right] at (x22) {$\whX_{2}^{(2)}$};
\fill (0,-2) coordinate (x3) circle (2pt); \node [right] at (x3) {$\whX_{3}$};
  \draw (x11) ..controls (-1.3,-.3) and (-1.3,-.7)  ..  (x21);
 \draw (x11) ..controls +(-40:1cm) and +(40:1cm)   ..   (x22);
  \draw (x12) ..controls (.7,-.3) and (.7,-.7)  ..  (x22);
 \draw (x12) ..controls +(-40:1cm) and +(40:1cm)   ..   (x21);
 \draw (x21) -- (x3);
  \draw (x22) -- (x3);
 \end{scope}
\end{tikzpicture}
 \caption{The level graphs $\overline{\Gamma}$, $\wh{\overline{\Gamma}}_{1}$, and $\wh{\overline{\Gamma}}_{2}$} \label{cap:GRC}
\end{figure}
\par
The \whgrc\ ($\widehat{4}$) for $\wh{\overline{\Gamma}}_1$ says that the
residue of $\whomega_{3}$ at each of the two nodes of $\wh{X}_3$ is zero. Since the sum of these two residues is always zero by the residue theorem for $\whomega_3$ on $\whX_3$, and the square of either residue of $\whomega_3$ is equal to the $2$-residue of $\eta_3$ at the node of $X_3$, it implies that condition ($\wh{4}$) for $\wh{\overline{\Gamma}}_1$ is equivalent to imposing $\Res^2\eta_3=0$. On the other hand, the \whgrc\ ($\widehat{4}$) for $\wh{\overline{\Gamma}}_2$ requires only that the sum of the two residues of $\whomega_3$ at the two nodes of $\wh{X}_3$ is equal to zero, which holds automatically by the residue theorem; the residues of $\whomega_3$ for this normalized cover may be non-zero, differing from each other by a sign.
\end{exa}
The rest of this section is devoted to the proof of Theorem~\ref{thm:kmain}.
We first  prove that the conditions are necessary. Then we develop an
equivariant version of plumbing techniques used in~\cite{plumb} to show that the conditions are also sufficient.
\par
\begin{proof}[Proof of Theorem~\ref{thm:kmain}, conditions are
necessary] We need to show that if a pointed
stable $k$-differential lies in the incidence variety compactification of a stratum, then it is associated
to some twisted $k$-differential satisfying conditions~(0)-(3) and~($\widehat{4}$).
Except for condition ($\widehat{4}$), the argument is similar to the case
of abelian differentials (see~\cite{plumb} for more details). We start
with a family of curves $f:\calX\to\Delta$ over a disk that degenerates to the central fiber~$X$,
with a family~$\xi_t$ of $k$-differentials of type $\mu$ for $t\in \Delta^{*}$, and let $Z_i$ be the
sections of marked zeros and poles.
Then for each irreducible component $X_v$ of $X$ there exists a scaling parameter $\ell_v\in\ZZ$ (in the sense of~\cite[Lemma~4.1]{plumb}) such that $\sum m_i Z_i$
is a section of the twisted relative $k$-dualizing sheaf $\omega_{f}^{k}
(\sum \ell_v X_v)$. For later use we make a degree $k$ base change such that $k\mid \ell_v$ for all $\ell_v$. Now the same proof as in the case $k=1$ leads to the necessary conditions (0), (1), and (2).
Hence we have constructed
a twisted $k$-differential $(X,\eta)$, where $ \eta_v := \lim_{t\to 0} t^{\ell_v}
\xi_t|_{X_v}.$ Moreover, the set of values $\{\ell_v\}$ defines a level graph $\overline\Gamma$, which
ensures that the partial order condition~(3) holds, again with the same
proof as for $k=1$.
\par
We now prove the necessity of condition~($\widehat{4}$). For $t \in \Delta^*$,
let $(\pi_t: \wh X_{t} \to X_t, \whomega_{t})$ be the family of normalized covers
associated with the family of $k$-differentials~$\xi_t$. It is well-known that the Hurwitz space of connected admissible covers is proper (see \cite[Paragraph~4]{hamu}). Since
a branched cover of degree $d$ is connected if and only if the associated monodromy group in the symmetric group $S_d$ acts transitively on the $d$ sheets of the cover, and the transitivity condition was not really used in the argument of properness, hence the Hurwitz space of possibly disconnected admissible covers is also proper.\footnote{Following a referee's suggestion, here is a more precise argument for the properness of the Hurwitz space of possibly disconnected admissible covers. Using the same notation, after a suitable base change for the family $\pi_t$ we may assume that the new family has no monodromy around $t = 0$ in the sense that the total space has exactly as many connected components as any of the special fibers. Then we can treat each component separately and apply the usual properness of the connected case. In particular, these components all extend uniquely to $t = 0$ after a possible further base change.}
It follows that the family $\pi_t$ extends over $t=0$ to some admissible
cover $\pi: \whX \to X$, and the order $k$ deck transformation $\tau_t$ of $\pi_t$ extends to a deck transformation $\tau$ of order $k$ for the cover $\pi$ in the central fiber. Then the scaling parameters for the family $\wh X_{t}$ of abelian differentials are equal to $\ell_{\pi(\whX_v)}/k$. It implies that the twisted abelian differential $\whomega$ obtained as the scaling limit of $\whomega_t$ is compatible with the lifted level graph
$\wh{\overline{\Gamma}}$. By the proof of necessity for abelian differentials given in~\cite[Section~4.1]{plumb}, we conclude
that~$\whomega$ satisfies the \whgrc\ imposed by $\wh{\overline{\Gamma}}$. To show that
$(\whX,\whomega,\tau)$ is a normalized cover of $(X,\eta)$, the only remaining properties
we need to justify are that $\tau^{\ast}\whomega=\zeta \whomega$
and that $\wh\omega^k = \pi^*\eta$. It suffices to check them on each irreducible component $X_{v}$ of $X$. The first property follows from taking the limit of the equations
$$\tau_{t}^{\ast}(t^{\ell_{v}/k}\wh\omega_{t})=t^{\ell_{v}/k}\tau_{t}^{\ast}(\wh\omega_{t})=t^{\ell_{v}/k}\zeta\wh\omega_{t} = \zeta (t^{\ell_{v}/k}\wh\omega_{t})$$
as $t$ goes to zero. The second property follows similarly from the fact that $\left(t^{\ell_{v}/k}\wh\omega_{t}\right)^{k}=t^{\ell_{v}}\pi_{t}^{\ast}(\xi_{t})$.
\end{proof}
The proof of sufficiency proceeds by passing to normalized covers and applying the
sufficiency for abelian differentials as in \cite{plumb}. The crucial ingredients of the proof for abelian differentials are
the existence of a modification differential and the procedure of merging nearby
zeros, which we now need to perform on an admissible cover in a way that is
equivariant under the action of the automorphism $\tau$.
\begin{lm}[Existence of equivariant modification differentials]\label{lm:moddiff}
Suppose that $\wh{Y}$ is a Riemann surface admitting an automorphism~$\tau$ of order $\ell>1$. For any $\tau$-invariant collection $\calQ$ of
points in~$\wh{Y}$
and any function~$r: \calQ \to \CC$ with $\tau^*r = \zeta  r$, where $\zeta$ is a primitive $\ell$-th root of unity,
 there exists
an abelian differential $\phi$ on~$\wh{Y}$ satisfying
the following properties:
\begin{itemize}
 \item[(i)]  The poles of $\phi$ are simple and contained in $\calQ$.
  \item[(ii)] The residue of $\phi$ at $q\in\calQ$ is $r(q)$.
 \item[(iii)]  The equivariance $\tau^{\ast}\phi=\zeta \phi$ holds.
\end{itemize}
\end{lm}
\par
\begin{proof}
Decomposing $\calQ$ into $\tau$-orbits and using $\sum_{i=1}^\ell \zeta^i = 0$
since $\ell>1$, we find that $\sum_{q \in \calQ} r(q) = 0$. Consequently, there is no
restriction imposed by the residue theorem, and hence by Mittag-Leffler theorem, there exists  a meromorphic abelian differential $\phi_0$ on~$\wh{Y}$ satisfying (i) and~(ii). To obtain the equivariance required in (iii), we simply average while twisting
by~$\zeta$, i.e., we let
$$\phi:=\frac{1}{\ell} \sum_{i=0}^{\ell-1}\zeta^{-i}(\tau^i)^*\phi_0.$$
Then $\phi$ satisfies (i) and (iii) by construction. Moreover, since $\tau^*r =
\zeta r$ we have
$$\Res_q(\zeta^{-1}\tau^*\phi_0)=\zeta^{-1}\Res_{\tau(q)}\phi_0
=\zeta^{-1}\zeta\Res_q \phi_0,$$
which implies that $\Res_q \phi=\Res_q \phi_0 = r(q)$. Hence $\phi$ also satisfies (ii).
\end{proof}
\par
\begin{proof}[Proof of Theorem~\ref{thm:kmain}, conditions are
sufficient] Our goal is to show that given a twisted $k$-differential $\eta$ and a
normalized cover $(\whX,\whomega)$ satisfying conditions (0)--(3)
and~($\widehat{4}$), there exists a family $f: \calX \to \Delta$ of
pointed stable curves, such that for $t\ne 0$, the curve $X_t$ is smooth and carries a
$k$-differential~$\xi_t$ of type~$\mu$, which converges to the pointed stable
$k$-differential~$\xi$ on~$X$ associated to $\eta$. The proof will moreover show
that $\eta$ is the scaling limit of $\xi$, in the sense of~\cite[Lemma~4.1]{plumb}
and of the proof of necessity above.
\par
By condition~($\widehat{4}$) and the main result of \cite{plumb}, the stable differential associated to the normalized cover $(\whX,\whomega)$ lies in the boundary of the incidence variety compactification
for the stratum of abelian differentials of type $\wh\mu$ (and if $\whX$ is disconnected we then apply \cite{plumb} to each connected component of $\whX$ separately). Hence there exists a family $\wh{f}:
\wh\calX \to \Delta$ of pointed stable curves, such that for $t\ne 0$, the curve $\whX_t$ is
smooth and carries an abelian differential~$\whomega_t$, while the scaling limit
of $\whomega_t$ is equal to $\whomega$. We need to show that the construction of the family
$\wh{f}$ can be carried out in such a way that the deck transformation $\tau$ on the central fiber $\wh{X}$
extends to an automorphism~$\tilde{\tau}$ of~$\wh\calX$ that restricted to each fiber is a deck transformation.
Then the desired family $\calX$ over $\Delta$ can be obtained by taking the quotient of $\whX$ under the extended automorphism~$\tilde{\tau}$, and
$\whomega^k$ on $\whX$ will be the pullback of a twisted $k$-differential from $X$ that satisfies conditions (0)-(3) and ($\wh 4$).
\par
We recall the outline of the proof given in~\cite{plumb} for abelian differentials and highlight the steps that need
modification in the equivariant setting.  The base of the induction, and in general
the final cleanup step, is to smooth all horizontal nodes by classical plumbing.
For this we need to make sure that~\cite[Proposition~4.4]{plumb} applies
equivariantly. Let $q$ be a horizontal node of $X$, and let~$\wh{q}$ be a preimage of $q$ on the normalized cover.
The upshot is that the $\tau$-orbit of~$\wh{q}$ consists of $k$ distinct points, so that we can plumb one of them and then move it around using~$\tau$. More precisely, having fixed a standard coordinate~$u$ for abelian differentials provided
in loc.\ cit., the plumbing (with fixture $\Omega_t =r(t) du/u$)
of the node $\wh{q}$ depends only on the residue~$r(t)$. Hence
we choose this standard coordinate $u$ at $\wh{q}$ and plumb there. Then
at a node $\tau^i(\wh{q})\in\whX$, which is also a preimage of the node $q$, we use $u\circ\tau^i$ as the standard
coordinate and apply $\tau^i$ to the previous plumbing construction, where $\tau$ acts on $\Omega$ by multiplication by $\zeta$. The resulting plumbing of $\whX$ is then $\tau$-equivariant by construction.
\par
The induction step consists of joining a smoothing
$f: \wh\calY \to \Delta$ of the (possibly disconnected) subcurve~$\wh{X}_{>L}$
of components of level greater
than~$L$ to the subcurve $\wh{X}_{=L}$. The differentials
are scaled by the factors of $t^\ell$ for a level function $\ell: V(\wh{\Gamma}) \to \ZZ_{\leq 0}$, which is inductively
constructed so that the differences between the levels are sufficiently divisible.
To apply higher order plumbing \cite[Theorem~4.5]{plumb}, it requires the existence of a modification
differential $\phi$ (denoted by $\xi$ in loc.\ cit.) with the following properties.
For every connected component~$\wh{Y}$ of~$\wh{X}_{>L}$, let $\{\wh{q}_{i,j}\}$ denote the set of nodes of $\wh{Y}$ where~$\wh{Y}$ intersects $\wh{X}_{=L}$, indexed in such a way that $\wh{q}_{i,j}$ is a preimage of the node $q_i$ of $Y$, and such that
$\tau(\wh{q}_{i,j}) = \wh{q}_{i, j+1}$, where the index $j$ is set to be modulo an appropriate
divisor of~$k$. Then the requirement on $\phi$ in loc.\ cit. is that
$\Res_{\wh{q}_{i,j}^+}(\phi) = - \Res_{\wh{q}_{i,j}^-}(\whomega)$. Note that if $k\mid \ord_{q_i}\eta$, then the $\tau$-orbit of $\wh{q}_{i,j}$ has cardinality~$k$. If $k\nmid \ord_{q_i}\eta$, then Proposition~\ref{prop:standard_coordinates} implies that
$\Res_{\wh{q}_{i,j}^-}(\whomega) = 0$. Therefore, if the $\pi$-preimage of a connected component of $X_{>L}$ is a disjoint union of $k$ curves, we can directly apply \cite[Lemma~4.6]{plumb} to obtain the desired modification differential $\phi$. On the other hand, if this preimage has $s<k$ connected components, we can use Lemma~\ref{lm:moddiff} with $\calQ$ given by the set of nodes where $\Res_{\wh{q}_{i,j}^-}(\whomega)\neq 0$, with the function $r(\wh{q}_{i,j}^+)=-\Res_{\wh{q}_{i,j}^-}(\whomega)$, with $\ell=k/s$ and with $\tau^s$ as the automorphism on each connected component, in order to obtain the modification differential $\phi$.
\par
The higher order plumbing for the node $\wh{q}_{i,j}$ is performed on some neighborhoods $\calD^{\pm}$
of $\wh{q}_{i,j}^\pm$, where we plumb $\psi_t^+ = t^c(\whomega_t^{+} + t^b\phi_t)$ to
$\psi_t^- = t^{b+c} \whomega_t^-$ using the plumbing fixture differential $\Omega_t$
given in~\cite[Theorem~4.5]{plumb}. Here $b$ and $c$ are constants depending on
$\ord_{q_{i,j}^+} \whomega$ and on the level function. On the neighborhoods $\tau^{-1}\calD^{\pm}$
of $\wh{q}_{i,j-1}^\pm$, consider the differentials
$\tau^*\psi_t^+ = \zeta \psi_t^+$ and $\tau^*\psi_t^- = \zeta t^{b+c}
\whomega_t^-$, where the equalities are implied by $\tau$-equivariance of $\phi$.
These can be glued using the plumbing fixture differential $\zeta\,\Omega_t$.
\par
More precisely, we simplify the notation by writing $q$ for a node of $X$ and let $m = \ord_{q^{+}}\eta$.
If $k\mid m$, the gluing of plumbing fixtures over $q$ can clearly be done as described above. Suppose $k\nmid m$ and let  $(u,v)$ be local (standard) coordinates around the node $q$ such that $\eta=u^{m}(du)^k$ and $\eta=v^{-m-2k}(dv)^k$. Recall that there are $r=\gcd(m,k)$ preimages of the node $q$ and that the order $\wh m$ of $\whomega$ at the preimages is equal to $\tfrac{m+k}{r}-1$. Consider the families of degenerating cylinders $\left\{(t,u_{j},v_{j})\ |\ u_{j}v_{j}=t^{d}\right\}$ for some $d$, with indices $j = 0, \ldots, r-1$, along with the forms $\Omega_{t}=t^{c}\zeta^{j}u_{j}^{\wh m}du_{j}$ and $\Omega_{t}=t^{b+c}\zeta^{j}v_{j}^{-\wh m-2}dv_{j}$. Define the action $\tau'(t,u_{j},v_{j})=(t, u_{j+1},v_{j+1})$ for $j=0,\dots,r-2$, and $\tau'(t,u_{r-1},v_{r-1})=(t, \zeta_{a}^{-r}u_{0},\zeta_{a}^{r}v_{0})$ for some $k$-th root of unity $\zeta_{a}$ to be determined. The induced action $(\tau')^*$ on $\Omega_{t}$ is clearly multiplication by $\zeta$ for $j\neq0$.  For $j=0$, the action on $\Omega_{t}$ is given by $(\tau')^{\ast}\left(t^c u_{0}^{\wh m}du_{0}\right)=\zeta_{a}^{m+k}\left(t^c u_{r-1}^{\wh m}du_{r-1}\right)$. In order to make it equal to $\zeta \Omega_t$ for $j = r-1$, we choose $\zeta_{a}$ such that $\zeta_{a}^{m+k}=\zeta^{r}$, which has a solution among $k$-th roots of unity since $\gcd(m, k)=r$. Similarly one can check it for the $v$ coordinates. Hence the action $\tau'$ extends $\tau$ on the families of degenerating cylinders. Consequently the family $\wh{f}_{\rm zeros}$
of pointed stable differentials obtained by plumbing admits the desired automorphism~$\tilde{\tau}$ that extends~$\tau$.
\par
The family of differentials $\wh{f}_{\rm zeros}$ does not yet have the
right type~$\wh{\mu}$, because a multiple zero $\wh z$ of~$\whomega_t$ may break into a collection of zeroes, of the same total multiplicity, when a small modification differential $\phi$ is added. In the case of abelian differentials, for each zero $\wh{z}$ of~$\whomega$ in the smooth locus of
$\wh{X}$ this issue is settled by merging zeros~\cite[Lemma~4.7]{plumb}, which is a local operation in a neighborhood $\calD$ of the
center of masses $\wh{x}(t)$ of the zeros that $\wh{z}$ splits to. It consists of replacing
the family of differentials $\psi_t = t^c(\whomega_t^+ + t^b\phi_t)$ by another family of
differentials~$\psi'_t$ with a unique zero of
multiplicity equal to the sum of multiplicities of the dispersed zeros. As in the step of higher order plumbing, we can perform this merging locally near one preimage $\wh{z}$ of a dispersed zero, and then use the property $\tau^* \psi_t = \zeta  \psi_t$ to replace on $\tau^{-1}(\calD)$ the differential $\zeta  \psi_t$ by $\zeta \psi'_t$. This can be done as in the previous paragraph in such a way that $\tau^* \psi'_t = \zeta  \psi'_t$. At the end we obtain a family $\wh{f}$
of pointed stable differentials of the correct type and still with an
automorphism~$\tilde{\tau}$ extending~$\tau$ as desired.
\end{proof}


\section{The global $k$-residue condition}
\label{sec:admtokGRC}

Throughout this section, we fix a nodal curve $X$ with a level graph $\overline\Gamma$ given by a level function~$\ell$. Whenever we take normalized covers of $X$ or of its subcurves, we implicitly consider them endowed with a lifted full order provided by
Lemma~\ref{lemma:liftcompatibleorder}.
\par
Previously we have characterized the incidence variety compactification of the strata of $k$-differentials in terms of the existence of a normalized cover satisfying the \whgrc~($\widehat{4}$). As Example~\ref{ex:GRCdependLift} demonstrates, the choice of a normalized cover is fundamental in the \whgrc~($\widehat{4}$). In this section we investigate the combinatorics of possible choices of  a normalized cover. By going through the various possibilities, eventually we will establish the equivalence of the \whgrc~($\widehat{4}$) and the global $k$-residue condition (4) in Definition~\ref{def:GRCk}.
\par
In order to establish the equivalence between conditions ($\widehat{4}$) and (4), the main tool we use is the following result.
\begin{prop}\label{prop:fewercomponents}
Let $(X,\eta)$ be a twisted $k$-differential.
If the restriction $\eta|_Y$ to a connected component~$Y$ of~$X_{>L}$
has a normalized cover $(\wh{Y}, \tau, \whomega)$ with strictly fewer than $k$ connected components,
then the global residue condition $(4$-{\rm ab}$)$ for abelian differentials in Definition~\ref{def:twistedAbType} imposed by~$\wh{Y}$ on $\whomega$ is automatically satisfied.
\end{prop}
\par
\begin{proof}
Let $\wh Y_1$ be one of the $r<k$ connected components of $\wh{Y}$. Denote by $q_{i}$ the nodes connecting~$Y$ to $X_{=L}$, and let $\wh{q}_{i, j}$ be the nodes connecting $\wh Y_1$ to $\whX_{=L}$. Since the cover $\whY\to Y$ is cyclic, the deck group must act transitively on the set of connected components of $\whY$.
It thus follows that $r$ divides $k$, and then that $\tau^{k/r}$ acts on $\wh{Y}_1$, and since $\tau^{\ast}(\whomega)=\zeta \whomega$, the sum of the residues of~$\whomega$ at the nodes $\wh{q}_{i,j}^{\,-}$ is proportional to $\sum_{j=1}^{k/r}\zeta^{jr}$, which is zero because the total sum of $(k/r)$-th roots of unity is zero for $r<k$.
\end{proof}
In order to apply the above proposition, we need to know when the preimage of a connected component of $X_{>L}$ can have fewer than $k$ connected components. The easiest case is described as follows.
\begin{lm}
Suppose the twisted $k$-differential $\eta_v$ on an irreducible component $X_v$ is not the $k$-th power of an abelian differential. For any lower level $L<\ell(v)$, let $Y$ be the connected component of $X_{>L}$ containing $X_{v}$. Then any normalized cover of $(Y,\eta|_Y)$ has strictly fewer than $k$ connected components.
\end{lm}
\begin{proof}
Let $(\whY,\tau,\whomega)$ be a normalized cover of $(Y,\eta|_Y)$. Then the restriction $\whX_v\to X_v$ of $\whY\to Y$ is a canonical normalized cover of $(X_v,\eta_v)$. Since $X_v$ is not the $k$-th power of an abelian differential, $\whX_v$ has strictly fewer than $k$ connected components. Since, as above, the deck group of $\whY\to Y$ acts transitively on the set of connected components of $\whY$, it follows that every connected component of $\whY$ contains at least one connected component of $\whX_v$, and hence $\whY$ also has strictly fewer than~$k$ connected components.
\end{proof}
\par
As a result, we show that if either case i) or case ii) in the global $k$-residue condition above holds, then there is no extra residue condition imposed.
\par
\begin{cor}\label{cor:iandii}
Suppose a connected component $Y$ of $X_{>L}$ contains an irreducible component~$X_v$ such that either $\eta_v$ is not the $k$-th power of an abelian differential, or $X_v$ contains a marked pole, i.e., there exists $z_i\in X_v$ with $m_i<0$. Then for any normalized cover of $(X,\eta)$ the global residue condition $(4$-{\rm ab}$)$ for abelian differentials imposed by $\wh Y$ on $\wh\omega$ is automatically satisfied.
\end{cor}
\par
\begin{proof}
If $\eta_v$ is not the $k$-th power of an abelian differential, then by the above lemma any
normalized cover of $(Y, \eta|_Y)$ has fewer than $k$ connected components. Hence condition $(4$-{\rm ab}$)$ imposed by $\wh Y$ is automatically satisfied by Proposition~\ref{prop:fewercomponents}.
\par
If~$X_v$ contains a marked pole~$z_i$ of order $m_{i}$, we distinguish two cases. If $k\;|\;m_i$,  it means that any abelian differential $\whomega_v$ such that $\whomega_v^{k}=\pi^{\ast}\eta_v$ has a pole of order $m_i/k$ at each of the~$k$ marked preimages of $z_i$. Then each connected component of $\whY$ contains a marked pole of the differential $\whomega_v$, in which case there is no global residue condition ($4$-ab) imposed by that connected component. On the other hand if $k\nmid m_i$, then $\eta_v$ is
not the $k$-th power of an abelian differential, and we are back to the beginning situation of the proof.
\end{proof}
\par
We now suppose that for every irreducible component
$X_v$ of $Y$, the twisted $k$-differential~$\eta_v$ is the $k$-th power of
an abelian differential, which is holomorphic away from the nodes.
It remains to determine under what conditions such a connected component $Y$
of $X_{>L}$ imposes a non-trivial \whgrc\ on a normalized cover.  In this situation,
for any normalized cover $(\whY,\tau,\whomega)$ of $(Y,\eta)$, the preimage
of any $X_v$ consists of $k$~isomorphic copies of $X_v$. We label these
components by $\whX_v(0), \ldots, \whX_v(k-1)$ in such a way that $\tau$ sends $\whX_v(i)$
to $\whX_v(i+1)$ and $\whX_v(k-1)$ to $\whX_v(0)$. Then the normalized cover has the
property that $\whomega_v(i+1)=\zeta\whomega_v(i)$, where $\zeta$ is the
chosen primitive $k$-th root of unity and $\whomega_v(i) = \whomega|_{\whX_v(i)}$.
\par
We first consider the case of horizontal nodes. Let $X'$ be a connected component of $X_{=K}$  for $K>L$. Note that there is a {\em unique} normalized cover $(\wh{X}',\whomega)$ of $(X',\eta|_{X'})$ which satisfies the matching residues condition (condition~(2) of Definition~\ref{df:twdk}) at the poles of order $k$, as the way to attach the preimages of the nodes is determined by this condition at the poles of order~$k$. Hence the fact that the normalized cover of $(X',\eta|_{X'})$ has fewer than~$k$ components can be characterized in the following way.
\begin{lm}\label{lm:horizontalnodes}
Let $X'$ be a connected component of $X_{=K}$ such that for every irreducible component~$X_v$ of $X'$, the twisted $k$-differential~$\eta_v$ is the $k$-th power of
an abelian differential. Then the twisted $k$-differential  $\eta|_{X'}$ is the $k$-th power of a twisted abelian differential if and only if the normalized cover $\wh{X}'$ has $k$ connected components.
\end{lm}
In order to make this condition more clear, we discuss an example as follows.
\begin{exa}[Horizontal criss-cross]
\label{ex:HorCrCr}
Let $k=2$ and consider a twisted $2$-differential obtained in the following way. We start from an abelian differential $(X,\omega)$ where $\omega$ has four simple poles $p_1, p_2, q_1$ and $q_2$. Suppose that the residues of $\omega$ at $p_{1}$ and $q_{1}$ is $1$ and at~$p_{2}$ and $q_{2}$ is $-1$. We obtain a twisted $2$-differential $(\bar X,\eta)$ by squaring $\omega$ and gluing~$p_1$ with~$q_{1}$ and $p_{2}$ with $q_{2}$. Clearly it satisfies condition~(2) from Definition~\ref{df:twdk} at every node, as the $2$-residues are all $(\pm 1)^2 = 1$.
The canonical cover of $(X,\omega^{2})$ has two components~$X_{\pm}$. In order to glue the nodes of these components
to a normalized cover, we have no choice in the gluing
but to pair the preimages $p_{i}^{+}$ of $p_{i}$ in $X_{+}$ with~$q_{i}^{-}$. Hence the only possible normalized cover, represented in Figure~\ref{cap:HorCrCr}, has only one connected component. On the other hand, if the twisted $2$-differential we started with had been obtained by identifying $p_{1}$ with $p_{2}$ and $q_{1}$ with $q_{2}$, then the normalized cover would have two connected components.
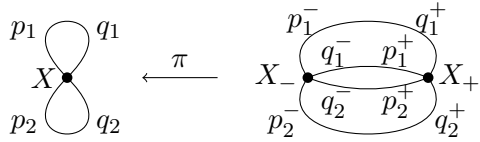
\begin{figure}[ht]
\centering
\begin{tikzpicture}[scale=1]
\fill (0,0) coordinate (x1) circle (2pt); \node [left] at (x1) {$X$};
 \draw (x1) ..controls (-1,1) and (1,1)  ..  (x1) coordinate[pos=.3](p1)coordinate[pos=.7](q1);
  \draw (x1) ..controls (-1,-1) and (1,-1)   ..  (x1) coordinate[pos=.3](p2)coordinate[pos=.7](q2);
  \node [left] at (p1) {$p_{1}$};  \node [right] at (q1) {$q_{1}$};
    \node [left] at (p2) {$p_{2}$};  \node [right] at (q2) {$q_{2}$};
  \begin{scope}[xshift=4cm]
  \draw[->] (-2,0) -- (-3,0);\node[above] at (-2.5,0) {$\pi$};
  \fill (-.8,0) coordinate (x1-) circle (2pt); \node [left] at (x1-) {$X_{-}$};
 \fill (.8,0) coordinate (x1+) circle (2pt); \node [right] at (x1+) {$X_{+}$};
 \draw (x1-) ..controls ++(-.7,-1) and ++(.7,-1)  ..  (x1+)coordinate[pos=.25](p2-)coordinate[pos=.75](q2+);
  \draw (x1-) ..controls ++(.5,-.2) and ++(-.5,-.2)   ..  (x1+)coordinate[pos=.25](q2-)coordinate[pos=.75](p2+);
  \draw (x1-) ..controls ++(-.7,1) and ++(.7,1)  ..  (x1+)coordinate[pos=.35](p1-)coordinate[pos=.65](q1+);
  \draw (x1-) ..controls ++(.5,.2) and ++(-.5,.2)   ..  (x1+)coordinate[pos=.25](q1-)coordinate[pos=.75](p1+);
    \node [left,yshift=.1cm] at (p1-) {$p_{1}^{-}$};  \node [yshift=.25cm] at (q1-) {$q_{1}^{-}$};
    \node [left] at (p2-) {$p_{2}^{-}$};  \node [yshift=-.2cm] at (q2-) {$q_{2}^{-}$};
    \node [yshift=.25cm] at (p1+) {$p_{1}^{+}$};  \node [right,yshift=.1cm] at (q1+) {$q_{1}^{+}$};
    \node [yshift=-.2cm] at (p2+) {$p_{2}^{+}$};  \node [right] at (q2+) {$q_{2}^{+}$};
    \end{scope}
\end{tikzpicture}
 \caption{The dual graphs of $(\bar X,\eta)$ and its normalized cover} \label{cap:HorCrCr}
\end{figure}

\end{exa}

\par
For vertical nodes, the construction of a normalized cover depends on identifying the preimages of the nodes, as described in the preceding section. If we choose the normalized cover $\wh Y\to Y$ such that it consists of $k$ disjoint copies $\whY(0),\ldots,\whY(k-1)$  isomorphic to $Y$, then the \whgrc{}s imposed by these copies on the cover differ by multiplication by powers of $\zeta$. The residue condition imposed by $\whY(0)$ is
\begin{equation}\label{eq:kres}
\sum_{q_j\in Y\cap X_{=L}}r_{q_j^-}=0,
\end{equation}
for some numbers $r_{q_j^-}$ such that $(r_{q_j^-})^k=\Resk_{q_j^-}\eta_{v^-(q_j)}$. Here the particular choice of a $k$-th root $r_{q_j^-}$ of the $k$-residue depends on which preimages of $q_j^-$ on the normalized cover $\whX_{=L}$ are connected to the preimages of the nodes $q_j^+$ lying on $\whY(0)$. Since we can first choose these preimages arbitrarily, and then identify preimages of $q_j^+$ in the other copies $\whY(i)$ with the preimages of $q_j^-$ on $\wh X_{=L}$ in a $\tau$-equivariant way, we have proven one implication of the following lemma for the construction of a normalized cover.
\par
\begin{lm} \label{le:kcopiesGRC}
Let $Y$ be a connected component of $X_{>L}$. If there is a normalized cover $(\wh Y,\whomega)$ of $(Y,\eta)$ satisfying the global residue condition (4-ab) which is the disjoint union of $k$ copies of $Y$, then there exists a normalized cover of $(Y\cup X_{=L},\eta)$ which coincides with $(\wh Y,\whomega)$ over $Y$ and satisfies the condition (4-ab)
 if and only if Equation~\eqref{eq:kres} holds for some $k$-th
roots $r_{q_j^-}$ of $k$-residues of $\eta$ at the nodes $q_j \in Y\cap X_{=L}$.
\end{lm}
\par
\begin{proof} In the above discussion we have verified the ``only if'' part.
Conversely, suppose that Equation~\eqref{eq:kres} holds for some $r_{q_j^-}$. Then the
residues of the $k$-th roots $\whomega_v$ of $\eta_v$ on the canonical
cover $\widehat{X}_{=L}$ at the nodes $\wh q_{j,i}^{\,-}$ lying over $q_j^-$ are
$\zeta^i r_{q_j^-}$ for $i=0,\ldots,k-1$, if the nodes are numbered
appropriately.
In this numbering we
construct the normalized cover of $Y\cup X_{=L}$ by connecting the
nodes $\wh q_{j,i}^{\,-}$ to the component $\whY(i)$ for all~$j$. Equation~\eqref{eq:kres}
is then equivalent to the global residue condition imposed by $\whY(0)$, and
more generally $\whY(i)$ imposes the $\tau^{i}$-image of this equation.
We thus conclude that our gluing defines a normalized cover of $\eta$
over  $Y\cup X_{=L}$ satisfying the condition (4-ab).
\end{proof}
\par
We can now conclude the case when a normalized cover contains a disjoint union of $k$ copies of $Y$.
\par
\begin{cor}\label{cor:trivialcover}
Let $Y$ be a connected component of $X_{>L}$.
If the normalized cover $(\wh Y,\whomega)$ of $(Y,\eta)$ is the disjoint union of $k$ copies of $Y$ and satisfies the condition (4-ab), then there exists a normalized cover over $Y\cup X_{=L}$ extending $(\wh Y,\whomega))$ which satisfies the \whgrc\ imposed by $Y$ if and only if Equation~\eqref{eq:GRCk} for the $k$-residues is satisfied.
\end{cor}
\begin{proof}
By definition of the polynomial $P_{n,k}$, Equation~\eqref{eq:GRCk} is simply the product of Equations~\eqref{eq:kres} for all possible choices of $k$-th roots of the relevant $k$-residues. Hence one of those sums is zero if and only if their product is zero.
\end{proof}
\par
It remains to deal with the case when for every $K>L$ on every connected component $Y'$ of $Y_{=K}$ the twisted $k$-differential $\eta|_{Y'}$ is the $k$-th power of a twisted abelian differential $\omega_{v}$, no $X_v$ contains a marked pole, but there exists a normalized cover $\whY$ of $Y$ satisfying the global residue condition for abelian differentials (4-ab) that is not a disjoint union of $k$ copies of $Y$. Such a phenomenon occurs when we identify the irreducible components of $\wh Y$ at the nodes in such a way that there is some crossing of the dual edges which cannot be straightened back (see $\wh{\overline{\Gamma}}_{2}$ in Example~\ref{ex:GRCdependLift}). For this reason, we call such a construction a {\em criss-cross}.
\par
Suppose that $Y$ has such a normalized cover $\whY$ and denote then by $K$ the lowest level such that $\whY$ restricted over $Y_{>K}$ consists of $k$ disjoint copies of $Y_{>K}$, while $\whY$ restricted over some connected component $W$ of $Y_{\ge K}$ has fewer than $k$ connected components.
By assumption, for each connected component $T_{\ell}$ of $W_{>K}$, the restriction of $\whY$ over $T_{\ell}$ consists of $k$ disjoint copies $\wh T_{\ell}(i)$ of $T_{\ell}$. Similarly we denote by $U_{s}$ the connected components of $W_{=K}$. Since we suppose that the restriction $\eta|_{U_{s}}$ is the $k$th power of a twisted abelian differential, there are $k$ disjoint preimages of $U_{s}$ in the normalized cover (see Lemma~\ref{lm:horizontalnodes}). We denote by $\wh U_{s}(i)$ these $k$ preimages. We denote $\whomega_s(i):=\zeta^i\omega_s$ the restriction of  $\pi^{\ast}(\eta)$ to $\wh U_{s}(i)$. Moreover, denote by $\Gamma_{{\rm red}}$ the level graph with two levels, whose vertices are the connected components $T_\ell$ of $W_{>K}$ and $U_s$ of $W_{=K}$, and whose edges are the nodes between these components. Note that this graph has no horizontal edges. The graph $\wh\Gamma_{{\rm red}}$ is defined analogously by adding ``hats'' in the preceding definition.
\par
Since by hypothesis $\wh W$ has fewer than $k$ connected components, there exists a closed loop $\gamma$ in $\Gamma_{{\rm red}}$ such that it has a lift $\wh\gamma$ to $\wh\Gamma_{\rm red}$ which is not a closed loop (and then by $\tau$-equivariance it follows that no lift of $\gamma$ is a closed path).
Given an edge $e$ of $\gamma$ connecting $T_{\ell}$ to $U_{s}$, its lift $\wh e$ starting at $T_{\ell}(0)$ goes to some $U_{s}(o(e))$ and we denote $\zeta_{e}:= \zeta^{o(e)}$. In the sum~\eqref{eq:kres} for $Y=T_{\ell}$, the residue at the node $q_{j}$ corresponding to the edge $e_{j}$ is then given by $\zeta_{e_{j}}\Res_{q_{j}^{-}}\omega_{u}$.
More generally, it follows from the $\tau$-equivariance that for a lift $\wh e$ of $e$ connecting the component $\wh{U}_{s}(a)$ to $ \wh T_{\ell}(b)$ we have $\zeta^{a-b}=\zeta_{e}$.
Applying this observation recursively along $\gamma$, we see that the product $\prod_{e\in\gamma}\zeta_{e}^{\pm 1}$ with alternating negative and positive exponents is a power of $\zeta$ whose exponent $i_{\gamma}$ gives the difference between the starting and the ending component of $\wh{U}_{\ell_{1}}$ of any lift $\wh\gamma$ of $\gamma$. By the hypothesis that $\wh \gamma$ is not closed, this exponent $i_{\gamma}$ is not zero modulo $k$, and hence the product of the roots with alternating exponents $\pm 1$ in the above is not equal to $1$.
\par
Now for the converse, suppose that there exists a connected component $W$ of $X_{\geq K}$ such that Equation~\eqref{eq:kres} is satisfied by taking $Y=T_{\ell}$ where $T_{\ell}$ are the connected components of $W_{>K}$. Suppose moreover that there exists a loop $\gamma\in\Gamma_{{\rm red}}$ such that the alternating product $\prod_{e\in\gamma}\zeta_{e}^{\pm1}$ is not equal to $1$. Then as explained in Lemma~\ref{le:kcopiesGRC}, we can construct a normalized cover of $(W,\eta|_{W})$ satisfying condition (4-ab) by identifying the preimages of the nodes between the components~$T_{\ell}$ and $W_{=K}$ in the way given by Equation~\eqref{eq:kres}. As explained in the previous paragraph, the assumption that the alternating product of the roots along $\gamma$ is not $1$ implies that the lifts of~$\gamma$ in $\wh\Gamma_{{\rm red}}$ are not closed paths. Hence there is a normalized cover with fewer than $k$ connected components above $W$, and
Proposition~\ref{prop:fewercomponents} can be applied.
\par
To conclude this discussion, we would like to state the condition for the loop $\gamma$ in the dual graph of $W$ directly (instead of in $\Gamma_{\rm red}$). This can be done as follows. To every path in the dual graph of $W$ we associate a path in $\Gamma_{\rm red}$ by contracting the edges between vertices of the connected components $T_{\ell}$ of $W_{>K}$. Conversely, to a path in $\Gamma_{\rm red}$ we associate a (not necessarily unique) path in the dual graph of $W$ such that it has the same edges joining the connected components~$T_{\ell}$ of $W_{>K}$ to the connected components $U_{s}$ of $W_{=K}$, and then connects arbitrarily within each~$T_{\ell}$. Clearly the condition on $\gamma$ can be reformulated by saying that there exists a loop $\gamma$ in the dual graph of $W$ such that the alternating product of the roots associated to the edges of $\gamma$ touching $W_{=K}$ is not equal to one.
\par
Summarizing the above discussion we have proved the following proposition.
\begin{prop}\label{prop:existconnectedcover}
Let $(Y,\eta)$ be a twisted $k$-differential satisfying condition (3) such that for every $K>L$ on every connected component $Y'$ of $Y_{=K}$ the twisted $k$-differential $\eta|_{Y'}$ is the $k$-th power of a twisted abelian differential $\omega_{v}$ without marked poles. Then there exists a normalized cover of $(Y,\eta)$ satisfying condition (4-ab) which is not a  disjoint union of $k$~copies of~$Y$ if and only if
  there exists a level~$K$, a connected component $W$ of $Y_{\geq K}$, a simple closed path $\gamma$ in the dual graph of $W$ such that  for every connected component $T_{\ell}$ of $W_{>K}$ and every edge $e$ in the set $E_{\ell}$  of nodes connecting $T_{\ell}$ to $Y_{=K}$, there exists a  $k$-th root of unity $\zeta_{e}$ satisfying
\begin{equation*}
   \sum_{e\in E_{\ell}} \zeta_e\Res_{q_e^-}\omega_{v^-(e)}\=0 \quad \text{ such that }    \quad   \prod_{e\in\gamma\cap E} \zeta_{e}^{\pm 1} \neq 1
\end{equation*}
 where $E=\bigcup_{\ell}E_{\ell}$ and the last product has alternating $\pm 1$ exponents along $\gamma$.
\end{prop}
\par
We are now ready to prove that the \whgrc\ ($\widehat 4$) is equivalent to the global $k$-residue condition (4).
\begin{proof}[Proof of the equivalence of $(\widehat{4})$ and $(4)$]
 We assemble the statements proven above to show that, for a twisted
$k$-differential, satisfying conditions (0)--(4)  is equivalent
to the existence of a normalized cover satisfying condition (4-ab).
\par
Suppose first that the global $k$-residue condition~(4) holds. We will construct the normalized
cover satisfying condition~($\widehat{4}$) inductively on the number of levels. The claim is empty, hence true, for twisted $k$-differential with one level. Suppose that the claim is true for every twisted $k$-differential with at most $|L|$ levels satisfying condition (3) (we prefer using negative levels, so that maximum is at level 0, so $L$ is negative). Let $(X,\eta)$ be a twisted $k$-differential with  $|L|+1$ levels satisfying conditions (3) and $(4)$, where the levels are $0,-1, \dots , L$ for simplicity. Let $Y_{1},\dots,Y_{r}$ be the connected components of $X_{>L}$. By assumption, for every $Y_{i}$ there exists a normalized cover $\wh Y_{i}$. If condition iii) holds for $Y_{i}$, then $\wh Y_{i}$ has fewer than $k$ connected components according to Lemma~\ref{lm:horizontalnodes}.  Moreover, if condition  iv) holds for $Y_{i}$, we may assume that $\wh Y_{i}$ has fewer than $k$ connected components, since Equations~\eqref{eq:cancellation1} and~\eqref{eq:cancellation2}
are exactly the hypothesis of  Proposition~\ref{prop:existconnectedcover}.
\par
Suppose that one of the first four conditions i), ii),  iii) or iv) in (4) holds for $Y_{i}$. Then by Corollary~\ref{cor:iandii} (for conditions i) and ii)) and Proposition~\ref{prop:fewercomponents} (for conditions iii) and iv)), any normalized cover of $(X,\eta)$ will satisfy the global residue condition for abelian differentials imposed by $\wh Y_{i}$. Now if none of these conditions holds for $\wh Y_{i}$, then condition v) must be satisfied. We use at $Y_{i}$ a normalized cover given in Corollary~\ref{cor:trivialcover}, which satisfies the global residue condition for abelian differentials at every preimage of every component $Y_{i}$. Hence we have constructed a  normalized cover of $X$ satisfying condition (4-ab), proving that condition $(\widehat{4})$ holds for $(X,\eta)$. Indeed we get a little more. If every component $Y_{i}$ has $k$ disjoint preimages and if the level $K$ in Proposition~\ref{prop:existconnectedcover} coincides with $L$, then there exists a  normalized cover of $X$ satisfying (4-ab) with fewer than $k$ connected components.
 \par
Conversely, suppose that the  \whgrc\ $(\widehat{4})$ holds. We will show inductively on the number of levels that condition (4) holds. If there is only one level, the claim is again empty, hence true. Suppose that the claim is true for every twisted $k$-differential with at most $|L|$ levels satisfying condition (3). Let $(X,\eta)$ be a twisted $k$-differential with  $|L|+1$ levels satisfying conditions (3) and $(\widehat{4})$. We denote by $(\whX,\wh\omega)$ a  normalized cover of $(X,\eta)$ given by condition $(\widehat{4})$. Now by induction, condition $(4)$ holds for every level $K>L$ and for every connected component of $X_{>K}$. Let $Y_{1},\cdots,Y_{r}$ be the connected components of $X_{>L}$. For each $Y_{i}$, we will analyze the global residue condition for abelian differentials (4-ab) induced by $\wh Y_{i}$ on $\wh\omega$.
\par
First the global residue condition can be void if the preimages $\wh Y_{i}$ of $Y_i$ contain a marked pole. Then the component $Y_{i}$ also contains a marked pole, and hence condition i) in (4) is satisfied. Next, the global residue condition for abelian differentials can be automatically satisfied if there are fewer than $k$ connected preimages of $Y_{i}$ in $\whX$. This can happen either because the restriction of $\eta$ on $Y_{i}$ is not the $k$-th power of an abelian differential or because there is some criss-cross gluing of nodes in $\whX$ over $Y_{i}$. In the former, condition ii) is satisfied. In the latter, either the criss-cross occurs with horizontal nodes and the hypothesis of Lemma~\ref{lm:horizontalnodes} is satisfied, leading to iii). Or the criss-cross occurs with vertical nodes and the hypothesis of Proposition~\ref{prop:existconnectedcover} must be satisfied for some level $K>L$,  leading to condition iv). If none of these cases happens, then the global residue condition for abelian differentials induced by $\wh Y_{i}$  is given by the vanishing of a factor of the polynomial defined in Equation~\eqref{eq:P}, which gives Equation~\eqref{eq:GRCk} of condition $(4)$. Hence condition $(4)$ holds at level $L$.
\end{proof}


\section{Dimension of spaces of twisted $k$-differentials}
\label{sec:dimension}

In this section we prove Theorem~\ref{thm:dimtwd}. Let $\overline{\Gamma}$ be a level graph and $h$ be the number of horizontal edges of $\overline{\Gamma}$. Recall that $\komoduliab$ is the union of the components of $\komoduli(\mu)$ consisting of $k$-th powers of holomorphic abelian differentials, and $\komodulinoab$ is the union of all the other components. We denote by $\tkdab$ (resp. $\tkdnoab$) the space of twisted $k$-differentials of type $\mu$ compatible with $\overline{\Gamma}$ which can be smoothed into $\komoduliab$ (resp. $\komodulinoab$), i.e., the corresponding twisted $k$-differential is a scaling limit of a one-parameter family in $\komoduliab$ (resp. $\komodulinoab$). We emphasize that the underlying dual graph of every point in $\tkdab$ (resp. $\tkdnoab$) is exactly $\Gamma$ and we allow global scaling of the differentials on the components $X_v$ corresponding to the vertices $v$ of $\Gamma$. For $k=1$, we simply use $\tdab$ to denote the space of twisted abelian differentials of type $\mu$ compatible with $\overline{\Gamma}$, which in the above notation is either $\mathfrak{M}^{1,\ab}(\overline{\Gamma})$ if $\mu$ is a holomorphic signature, or $\mathfrak{M}^{1,\nonab}(\overline{\Gamma})$ otherwise. We will verify the dimension count in Theorem~\ref{thm:dimtwd} in the following two steps, first for $k=1$ and then for $k\geq 2$.
\par
\begin{thm} \label{thm:dimabtwd} The space
$\tdab$ of twisted abelian differentials
is either empty or has dimension equal to $\dim \omoduli(\mu) - h$.
\end{thm}
\par
\begin{thm} \label{thm:dimtwkd}
For $k\geq 2$, the space
$\tkdnoab$ is either empty or has
dimension equal to $\dim \komodulinoab - h$.
\end{thm}
The same result holds by replacing ``non-ab''  with ``ab'' in Theorem~\ref{thm:dimtwkd}, whose proof simply reduces to the abelian case as in Theorem~\ref{thm:dimabtwd}, because any deformation of a $d$-th power of a $(k/d)$-differential remains to
be a $d$-th power as explained in and after Theorem~\ref{thm:DeformkDiff}.
\par
We do some preparations before proving these theorems. First, we can reduce the situation to $h = 0$, i.e., $\oG$ has no horizontal edges, since such a node is locally smoothable by classical plumbing, which affects the dimension count precisely by one. Hence from now on we assume that the level graph $\oG$ has the property that $h = 0$.
\par
Consider the case of abelian differentials. Take a pointed topological surface $\Sigma$ with $\Lambda\subset \Sigma$ as a disjoint union of simple closed curves such that the degenerate surface with dual graph $\Gamma$ is obtained by
pinching curves in $\Lambda$ to the corresponding nodes. Denote by $\Lambda^\circ$ an open thickening of $\Lambda$ with upper boundary $\Lambda^+$ and lower
boundary~$\Lambda^-$. Let $P\subset \Sigma\setminus \Lambda^\circ $ be the set of marked points corresponding to the interior poles of the differentials, and $Z\subset \Sigma\setminus \Lambda^\circ$ corresponding to the interior marked zeros.
\par
Compatibility of twisted abelian differentials with $\oG$ is governed by the global residue condition ($4$-ab), which is imposed to the residue assignments at the nodes of the degenerate surface. Moreover, such residue assignments should also satisfy the residue theorem on each irreducible component of the surface. Therefore, we define the following subspaces of residue assignments
$$ R^{\rm grc} \subseteq R^{\rm res} \subseteq \CC^{h^0(\Lambda)}, $$
where $R^{\rm res}$ is the subspace satisfying the residue theorem at the irreducible components of $\Sigma\setminus \Lambda$ that have no marked poles in the interior but have poles at some resulting nodes (i.e., those lower level components without poles in the interior and moreover not local maxima of $\oG$), and where $R^{\rm grc}$ is the subspace cut out further by the global residue condition for twisted differentials compatible with $\oG$. Let $c_{\rm L}$ be the number of irreducible lower level components without poles in the interior and not local maxima of $\oG$. Note that each component counted in $c_{\rm L}$ admits an independent condition for the constraint of the residue theorem, since every (polar) node belongs to a unique such component. It follows that
\be \label{eq:LLResThm}
h^0(\Lambda) - \dim R^{\rm res} \= c_{\rm L}.
\ee
Note also by our setting that the residue map
$$\tdab \to R^{\rm res}$$
factors through~$R^{\rm grc}$.
\par
Next we give an alternative form of the global residue condition ($4$-ab). Without loss of generality, assume that the $N$ levels of $\oG$ are labeled by zero for the top level and by $-i$ for the $i$-th level for $1\leq i \leq N-1$. Treating $\Lambda$ as the set of edges of $\oG$, define a subset $\Lambda_j \subseteq \Lambda$ by
$$ \Lambda_j \=
\bigl\langle  \lambda \in \Lambda \,:\,
\ell (v^-(\lambda)) \leq -j \bigr\rangle.
$$
Namely,  $\Lambda_j$ is the subset of simple closed curves in $\Lambda$ whose corresponding edges in $\oG$ have lower ends on level $\leq -j$.
Let $V_j\subset H_1(\Sigma\setminus P, Z)$ be the image subgroup generated by the classes of the simple closed curves in $\Lambda_j$ under the map $H_1(\Lambda) \to H_1(\Sigma\setminus P, Z)$ induced by the inclusion $\Lambda\hookrightarrow \Sigma$. We thus obtain
a filtration of linear subspaces
$$ 0\= V_N \,\subseteq \, V_{N-1} \,\subseteq\, \cdots \,\subseteq\, V_1. $$
Since any residue assignment to the nodes of the pinched surface corresponds to a function $\rho: \Lambda \to \CC$, we can characterize the subspace $R^{\rm grc}$ as follows.
\par
\begin{prop}  \label{prop:GRCviaHomo}
A residue assignment $\rho: \Lambda \to \CC$
satisfies the global residue condition if and only if there exist linear maps
$$ \rho_j: V_j/V_{j+1} \to \CC \quad \text{for $j=1,\ldots,N-1$,}$$
such that $\rho_j(\lambda) = \rho(\lambda)$ for all simple closed curves
$\lambda$ in $\Lambda$, where~$j$ is determined by $\ell(v^-(\lambda)) = -j$.
\par
In particular, $\dim R^{\rm grc} = \dim V_1 = \dim \Im (H_1(\Lambda) \to H_1(\Sigma\setminus P, Z))$.
\end{prop}
\par
\begin{proof}
Given the collection of $\rho_j$, for every~$j$ and every connected component $Y$ of $\oG_{>-j}$ without
marked poles in the interior, the sum of the residues on all (down-going) edges of $Y$ is
zero, as one can see by summing the corresponding boundary cycles
in $H_1(\Sigma \setminus P, Z)$ and the edges to levels below~$-j$
contribute zero since $\rho_j$ maps $V_{j+1}$ to zero. This zero summation is precisely
the global residue condition ($4$-ab). The converse is just a reformulation of the global residue condition.
The second statement about $\dim R^{\rm grc}$ follows from the first, since $\rho$ is uniquely specified by any collection of $\rho_j \in (V_j/V_{j+1})^{*}$ and $\sum_{j=1}^{N-1} \dim(V_j/V_{j+1}) = \dim V_1$.
\end{proof}
\par
\begin{proof}[Proof of Theorem~\ref{thm:dimabtwd}]
The stratum $\omoduli(\mu)$ is locally modeled on the periods of the relative
homology group $H_1(\Sigma \setminus P,Z)$.
The space of twisted abelian differentials compatible with $\oG$ except for the global residue condition is locally modeled
on the periods of the homology group
$H_1(\Sigma \setminus \{P \cup \Lambda^\circ\}, \Lambda^+ \cup Z)$ by
applying Corollary~\ref{cor:period} to all components of $\Sigma \setminus \Lambda^{\circ}$ separately.
Since the residue map from this space of twisted differentials to $R^{\rm res}$
is dominant and equidimensional over its image (being the composition of
a local biholomorphism and a linear projection under period coordinates), it suffices to show that
\be \label{eq:aim}
h_1(\Sigma \setminus \{P \cup \Lambda^\circ \}, \Lambda^+ \cup Z)
- h_1(\Sigma \setminus P,Z) \= \dim R^{\rm res} - \dim R^{\rm grc}.
\ee
\par
Take the long exact sequence of homology of the triple $Z \subset \Lambda \cup Z \subset \Sigma \setminus P$
\[\begin{tikzpicture}
\matrix (m) [matrix of math nodes, row sep=0.10em,
column sep=2.0em, text height=1.5ex, text depth=0.25ex]
{
H_2(\Lambda \cup Z, Z) = 0 & H_2(\Sigma \setminus P, Z) & H_2(\Sigma \setminus P, \Lambda \cup Z) \\
H_1(\Lambda \cup Z, Z)  & H_1(\Sigma \setminus P, Z)
& H_1(\Sigma \setminus P, \Lambda \cup Z)  \\
H_0(\Lambda \cup Z, Z) & H_0(\Sigma \setminus P, Z)=0. & \\
};
\draw[->,font=\scriptsize,every node/.style={above}, rounded corners]
(m-1-1) edge  (m-1-2)
(m-1-2) edge  (m-1-3)
(m-1-3.east) --+(5pt,0)|-+(0,-7.5pt)-|([xshift=-5pt]m-2-1.west)--(m-2-1.west)
(m-2-1) edge  (m-2-2);
\draw[->,font=\scriptsize,every node/.style={above}, rounded corners]
(m-2-2) edge  (m-2-3)
(m-2-3.east) --+(5pt,0)|-+(0,-7.5pt)-|([xshift=-5pt]m-3-1.west)--(m-3-1.west)
(m-3-1) edge  (m-3-2)
;
\end{tikzpicture}\]
Let $c \coloneqq h_2(\Sigma \setminus P, \Lambda \cup Z) = h^0(\Sigma \setminus \{\Lambda \cup Z\}, P)$ by duality (see~\cite[Theorem 6.2.17]{spanier}), which is the number of connected components of $\Sigma \setminus \Lambda$ without poles in the interior. Let $\delta_P \coloneqq h_2(\Sigma\setminus P, Z) = h^0(\Sigma\setminus Z, P)$, which is $1$ if $P = \emptyset$ and is $0$ otherwise. In addition, note that $h_i(\Lambda \cup Z, Z) = h_i(\Lambda) $ for $i=0,1$ and the same for their cohomology,
 all of which are equal to the number of simple closed curves in $\Lambda$, i.e., the number of edges of $\oG$. Also note that by
applying the excision theorem to $\Lambda^\circ$ we have
\begin{equation}\label{eq:excision}
 h_i(\Sigma \setminus \{P \cup \Lambda^\circ\}, \Lambda^+ \cup \Lambda^- \cup Z) = h_i(\Sigma \setminus P, \Lambda \cup Z).
\end{equation}
Using these observations we can deduce from the long exact sequence that
\be \label{eq:h1basic}
h_1(\Sigma \setminus \{P \cup \Lambda^\circ\}, \Lambda^+ \cup \Lambda^- \cup Z)
- h_1(\Sigma \setminus P,Z) \= c - \delta_P.
\ee
In addition, the exact sequence and Proposition~\ref{prop:GRCviaHomo} imply that
\be \label{eq:Rdim}
\dim R^{\rm grc} \= h_1(\Lambda \cup Z, Z) \,-\, (c -\delta_P) \= h^0(\Lambda)\,-\, (c -\delta_P).
\ee
\par
Note that~\eqref{eq:excision} implies that
$c=h_2(\Sigma \setminus  \{P \cup \Lambda^\circ\},  \Lambda^+ \cup \Lambda^-
\cup Z)$.
Combining this with the long exact sequence of homology of the triple
$ \Lambda^+ \cup Z \subset \Lambda^+ \cup \Lambda^- \cup Z \subset
\Sigma \setminus \{P \cup \Lambda^\circ\}$, we obtain that
\be \label{eq:h1bdconv}
h_1(\Sigma \setminus \{P \cup \Lambda^\circ\}, \Lambda^+ \cup \Lambda^- \cup Z)
- h_1(\Sigma \setminus \{P \cup \Lambda^\circ\}, \Lambda^+ \cup Z)
\= c - c_{\rm M},
\ee
where $c_{\rm M} \coloneqq h_2(\Sigma \setminus  \{P \cup \Lambda^\circ\},
\Lambda^+ \cup Z)
= h^0(\Sigma \setminus  \{Z \cup \Lambda^\circ\}, \Lambda^- \cup P)$
is the number
of connected components of $\Sigma \setminus \Lambda$ containing no poles,
neither in the interior nor at the nodes by pinching the (lower boundary) components of $\Lambda$. Said differently,
$c_{\rm M}$ is the number of pole-free local maxima of $\oG$.
\par
Using the obvious relation $c = c_{\rm M} + c_{\rm L}$, combining~\eqref{eq:h1basic} with \eqref{eq:h1bdconv} and comparing
this to the difference of~\eqref{eq:Rdim} and~\eqref{eq:LLResThm}, we thus obtain the desired~\eqref{eq:aim}.
\end{proof}
\par
As a consequence of the above calculation we also obtain that
$$\dim R^{\rm res} - \dim R^{\rm grc} \= c_{\rm M} - \delta_P.$$
Namely, the number of independent global residue conditions is precisely the number of pole-free
local maxima of~$\oG$, up to one global redundancy if the set of marked poles $P$ is empty.
\par
Now we prove the case for $k\geq 2$, which is essentially an equivariant version of the preceding proof.
\par
\begin{proof}[Proof of Theorem~\ref{thm:dimtwkd}]
We first give the setting which allows us to model the stratum. Fix a twisted $k$-differential $(X,\eta)$ and a connected  normalized cover $(\whX,\whomega,\pi)$ satisfying the global residue condition for abelian differentials (4-ab). The assumption that $\whX$ is connected is not a restriction, since otherwise we consider $(k/d)$-differentials where $d$ is the number of connected components of $\whX$, and any deformation of a $d$-th power of a $(k/d)$-differential remains to
be a $d$-th power.
Take two topological surfaces $\wh\Sigma$ and $\Sigma$ with two disjoint unions of simple closed curves $\whLa \subset \wh \Sigma$ and $\Lambda\subset \Sigma$, such that $\whX$ and
$X$ are obtained by pinching curves in $\whLa$ and $\Lambda$ respectively and such that $\pi$ lifts to a cover from $\wh\Sigma$ to $\Sigma$. Denote by $P$ the set of poles of $\eta$ with pole order $\geq k$ in the interior of $\Sigma\setminus \Lambda$.
Denote by $Z$ the set of zeros along with the remaining poles in the interior of $\Sigma\setminus \Lambda$. We denote similarly by $\whP$ and $\whZ$ the sets of preimages under $\pi$ of~$P$ and $Z$ in the interior of  $\wh\Sigma\setminus \whLa$, which correspond to the interior poles and zeros of the lifted twisted differential $\wh\omega$ respectively.
\par
The deck transformation $\tau$ on $\whX$ induces a deck transformation on $\wh\Sigma$, which we still denote by~$\tau$.
We denote by~$E_i$ the eigenspaces of various homology groups where
$\tau$ acts by the eigenvalue~$\zeta$ (the chosen primitive $k$-th root of unity)
and denote the dimension of $E_i$ by~$e_i$. Similarly we use $E^i$ and $e^i$ for the $\zeta$-eigenspaces of cohomology groups and their dimensions. Let $\Sigma_k$ be the set of connected components of $\Sigma \setminus
\Lambda$ where the twisted $k$-differential $\eta$ is the $k$-th power of an abelian differential
and let $\Lambda_k \subset \Lambda$ be the subset of curves where the vanishing orders of $\eta$ at the corresponding nodes
are a multiple of~$k$. Note that $\Sigma_k$ and $\Lambda_k$ are invariant under small deformations of $\eta$, as any deformation of a $d$-th power of a $(k/d)$-differential remains to be a $d$-th power.
\par
The stratum $\komoduli(\mu)$ is locally modeled on the periods of
$E_1(\wh\Sigma \setminus \whP,\whZ)$ as discussed in Section~\ref{sec:DefOfMeroDiff}.
The support of the eigenspace $E_i(\whLa)$ for $i=0, 1$ is on the preimage
$\whLa_k = \pi^{-1}(\Lambda_k)$ as the preimage of every simple closed curve in $\Lambda \setminus \Lambda_k$ has fewer than $k$ components. Since by Proposition~\ref{prop:fewercomponents} the \whgrc~only imposes conditions to nodes
that correspond to $\whLa_k$, we define spaces of residue assignments at such nodes
\bes
R^{\rm grc}_k \, \subseteq \, R^{\rm res}_k \, \subseteq \, \CC^{e^0(\whLa_k)}\,\subseteq \, \CC^{h^0(\whLa_k)}.
\ees
The subspace $\CC^{e^0(\whLa_k)}$ consists of residue assignments that are equivariant under the deck transformation $\tau$.
The subspace $R^{\rm res}_k$ is further cut out by imposing the residue theorem to the lifts of the components in~$\Sigma_k$ that are not local maxima of $\oG$ and have no poles of order $\geq k$ in the interior, whose number we denote by~$c_{k, {\rm L}}$. The subspace~$R^{\rm grc}_k$ satisfies moreover the global residue condition ($4$-ab). As in~\eqref{eq:LLResThm} we obtain that
$$e^0(\whLa_k) - \dim R^{\rm res}_k = c_{k, {\rm L}}.$$
\par
The space of (lifts to the normalized covers of) twisted $k$-differentials
compatible with~$\oG$ except for the \whgrc~is modeled
on $E_1(\wh\Sigma \setminus \{\whP \cup \whLa^\circ \}, \whLa^+ \cup \whZ)$.
Since the residue map from this space is dominant to~$R^{\rm res}_k$ and equidimensional over its image, it suffices to show that
\bes
e_1(\wh\Sigma \setminus \{\whP \cup \whLa^\circ\}, \whLa^+ \cup Z)
- e_1(\wh\Sigma \setminus \whP, \whZ) \= \dim R^{\rm res}_k - \dim R^{\rm grc}_k
\ees
as in~\eqref{eq:aim}.
\par
Since by assumption $\wh\Sigma$ is connected,
$h_2(\wh\Sigma \setminus \whP, \whZ) = h^0(\wh\Sigma \setminus \whZ, \whP)$ is $0$ or~$1$. Note that if it is~$1$, the nontrivial contribution comes from
the class of $\wh\Sigma$ itself, which is fixed under $\tau$. Hence the $\zeta$-eigenspace of $H_2(\wh\Sigma \setminus \whP, \whZ)$ is trivial, which implies that
$e_2(\wh\Sigma \setminus \whP, \whZ) = 0$.
\par
Now the rest of the calculation is the same as in the abelian
case, after decorating all symbols with hats and using respective eigenspaces. We thus conclude that
\bes
e_1(\wh\Sigma \setminus \{\whP \cup \whLa^\circ\}, \whLa^+ \cup \whLa^- \cup \whZ)
- e_1(\wh\Sigma \setminus \whP, \whZ) \= c_{k} \,
\ees
as in~\eqref{eq:h1basic},
where $c_{k} \coloneqq e_2(\wh\Sigma \setminus \whP, \whLa \cup \whZ)
= e^0(\wh\Sigma \setminus \{\whLa \cup \whZ\}, \whP)$ is the number
of components in $\Sigma_k$ without poles of order $\geq k$ in the interior, and the analogue of $\delta_P$ is $e_2(\wh\Sigma \setminus \whP, \whZ) = 0$ as shown above.
Then we have
$$\dim R^{\rm grc}_k \= e^0(\whLa_k)- c_{k}$$
as in~\eqref{eq:Rdim}, and moreover
\bes
e_1(\wh\Sigma \setminus \{\whP \cup \whLa^\circ\}, \whLa^+ \cup \whLa^- \cup \whZ)
- e_1(\wh\Sigma \setminus \{\whP \cup \whLa^\circ\}, \whLa^+ \cup \whZ)
\= c_{k} - c_{k, {\rm M}} \,
\ees
as in~\eqref{eq:h1bdconv},
where $c_{k, {\rm M}} \coloneqq e_2(\wh\Sigma \setminus  \{\whP \cup \whLa^\circ\},
\whLa^+ \cup \whZ)
= e^0(\wh\Sigma \setminus  \{\whZ \cup \whLa^\circ\}, \whLa^- \cup \whP)$
is the number of components in $\Sigma_k$ that are local maxima
of $\oG$ with no poles of order $\geq k$, neither in the interior nor at the boundary. Since $c_{k} =  c_{k, {\rm M}} +
c_{k, {\rm L}}$, we thus obtain the desired result as before.
\end{proof}
\par
As a consequence of the above calculation we also obtain the following
interpretation of the \whgrc, once we fix the set of components $\Sigma_k$
on which the $k$-differentials are $k$-th powers of abelian differentials:
$$ \dim R^{\rm res}_k - \dim R^{\rm grc}_k \= c_{k, {\rm M}}.$$
Namely, the number of independent \whgrc~is precisely the number $c_{k, {\rm M}}$ of local maxima
of $\oG$ in $\Sigma_k$ with no poles of order $\geq k$.



\section{Examples of flat geometric constructions with $k$-differentials}\label{sec:examples}

In this section we apply the incidence variety compactification to generalize certain flat geometric surgeries from abelian and quadratic differentials to $k$-differentials. Moreover, we sketch a flat geometric approach for proving Theorem~\ref{thm:kmain}.

\begin{exa}[Breaking up a singularity]
In this example, we use Theorem~\ref{thm:kmain} to extend the notion of breaking up a zero of abelian and quadratic differentials (see~\cite{kozo1, lanneauquad}) to the case of $k$-differentials. Here we describe the construction for singularities with $m>-k$. One can similarly extend this construction to the case of breaking up a pole of order $\leq -rk$ into $r$ poles of order $\leq -k$ or even more general settings.

Let $(X_{0},\xi)\in\komoduli(m,m_{1},\dots,m_{n})$ be a $k$-differential with a singularity $z$ of order $m > -k$ that we want to break into $r$ singularities of orders~$\widetilde m_{i}$, where $\sum_{i=1}^r \widetilde m_i = m $. Let $(X,\eta)$ be the twisted $k$-differential of type $(\widetilde m_{1},\dots,\widetilde m_{r},m_{1},\dots,m_{n})$, where $X$ is the union of $X_{0}$ with a $\PP^{1}$ attached at $z$, the restriction $\eta|_{X_0} = \xi$, and $\eta|_{\PP^1}$ is a $k$-differential with singularities of order $\widetilde m_{i}$ in the smooth locus of $\PP^{1}$ and with a pole of order $m+2k$ at the node $z$. If $(X,\eta)$ satisfies the global $k$-residue condition (4) in Definition~\ref{def:GRCk}, then the family of plumbed $k$-differentials in the proof of Theorem~\ref{thm:kmain} breaks $z$ into $r$ singularities of orders $\widetilde m_{1},\ldots, \widetilde m_{r}$, on smooth curves.

In this case condition (4) does not hold if and only if $\xi$ is the $k$-th power of a holomorphic abelian differential and every differential in the stratum $\komoduli[0](\widetilde m_{1},\dots,\widetilde m_{r},-2k-m)$ has non-zero $k$-residue at the pole of order $2k+m$. On the other hand if there exists a differential in the stratum $\komoduli[0](\widetilde m_{1},\dots,\widetilde m_{r},-2k-m)$ with zero $k$-residue at the pole of order $2k+m$, then the operation of breaking up such singularity can be performed locally. Otherwise we need to add a modification differential $\phi$ to $\xi$ as in the proof of Theorem~\ref{thm:kmain}, and hence the smoothing construction may fail to be local. For instance, this explains the failure of breaking up locally a zero of even degree into two zeros of odd degrees in the case of quadratic differentials as remarked in~\cite{mz}.

 A consequence of the above discussion is that if $k\nmid m$, then one can always break up locally a singularity of order $m$ into $r$ singularities of orders that add up to $m$.
 \end{exa}

\begin{exa}[Bubbling a handle]
Another useful construction developed for abelian and quadratic differentials is bubbling a handle (see \cite{kozo1, lanneauquad}). Given a quadratic differential $(X,q)$, this surgery increases the genus of $X$ by one and the order of a zero of $q$ by four, while preserving the orders of the other singularities. Here we generalize this operation to the case of $k$-differentials.

Let $(X_{0},\xi)\in\komoduli(m,m_{1},\dots,m_{n})$ be a $k$-differential with a singularity $z$ of order $m>-k$ at which we want to bubble a handle. Let $(X,\eta)$ be the twisted $k$-differential of genus $g+1$ and of type $(m+2k,m_{1},\dots,m_{n})$, where $X$ is the union of $X_{0}$ with a torus $X_{1}$ attached at $z$, the restriction $\eta|_{X_0} = \xi$, and $\eta|_{X_1}$ is a $k$-differential that has a zero of order $m+2k$ at some point of~$X_1$ and a pole of order $-m-2k$ at the node. If $(X,\eta)$ satisfies condition (4) in Definition~\ref{def:GRCk}, then the family of plumbed $k$-differentials obtained from $(X,\eta)$ gives the operation of bubbling a handle at $z$.

Note that there always exists a $k$-differential in the stratum $\komoduli[1](m+2k,-2k-m)$ which has zero $k$-residue at the pole. If $k\nmid m$, the claim follows from definition. If $k \mid m$, we can simply take the $k$-th power of an abelian differential of type $(m/k + 2, -m/k -2)$. Moreover, for every stratum different from $\Omega^{2}\moduli[1](4,-4)$, there exists a primitive differential in $\komoduli[1](m+2k,-2k-m)$ with vanishing $k$-residue at the pole (see \cite{geta}).  As a result, $(X,\eta)$ satisfies condition (4), and hence it is always possible to bubble a handle at such a singularity $z$ of a $k$-differential. This bubbling can furthermore be performed in such a way that the result is a primitive $k$-differential, even starting from the $k$-th power of an abelian differential --- the only exception being trying to bubble a handle at a regular point of the square of an holomorphic differential.
\end{exa}

\begin{exa}[Flat geometric smoothing]
As in the case of abelian differentials, sufficiency of the conditions in Theorem~\ref{thm:kmain} can also be explained by a flat geometric construction in the sense of $(1/k)$-translation structure, without passing to the canonical cover. We illustrate the idea via the following example. The reader may refer to~\cite[Paragraph~5]{plumb} for more details on related objects that will be used below.

Let $(X,\eta)$ be a twisted $k$-differential compatible with a given level graph with two levels, such that the top level contains a single component $(X_{0},\eta_{0})$, and the lower level consists of several components $(X_{a},\eta_{a})$. Suppose $\eta_{0}$ is the $(k/d)$-th power of a primitive $d$-differential. Let $q_{i}$ be all the vertical nodes of $X$, with $q_{i}^-\in X_{a(i)}$ in the lower level. On $X_{0}$ we want to insert a line segment representing $\Resk_{q_{i}^{-}} \eta$ such that its middle point is at $q_{i}^+$. By this, we mean that the segment is one of the $k$-th roots of $\Resk_{q_{i}^{-}} \eta$. If $d\neq k$, then we can choose any such roots. If $d=k$, we want to choose the roots such that the corresponding sum factor in Equation~\eqref{eq:cancellation1} vanishes. Such a choice exists by the global $k$-residue condition. We divide each of these segments into $d$ segments of equal length $\tilde{r}_{i,j}$, and label the middle point of $\tilde{r}_{i,j}$ by $q_{i,j}$ for $1\leq j\leq d$.

We next introduce the definition of a residue slit, as a collection of broken lines in the surface
around which we will modify neighborhoods of size given by the $k$-th roots of the $k$-residues.
Given the tuple $T = (\zeta^{jk/d}\tilde{r}_{i,j})_{\left\{i,j\right\}}=(r_{1}, \ldots, r_N)$ with cardinality $N$, we fix a permutation $\pi \in S_N$ such that the slopes of $r_{\pi(1)},
\ldots,r_{\pi(N)}$ are monotone. Let $P = P(T,\pi)$ be the polygon whose
edges are given by the vectors $r_{\pi(1)}, \ldots,r_{\pi(N)}$ consecutively. It follows
that $P$ is convex. Let~$B(P)$ be the barycenter of~$P$. Let~$p$ be a point in~$X_{0}$,
disjoint from the singularities of $\eta$. Shrinking $\eta_i$ if necessary, we place $P$ inside $X_0$ with $B(P) = p$.
A {\em $k$-residue slit} for~$(T,\pi)$ is then defined to be a collection of broken lines $\lbrace b_{i,j}\rbrace$ with
the following properties:
\begin{itemize}
\item Each broken line $b_{i,j}$ starts from $q_{i,j}$ and connects to the middle point of the edge $\zeta^{jk/d}\tilde{r}_{i,j}$ of the polygon~$P$. We denote by $b_{i,j}^{(\ell)}$ the $\ell$-th line segment of $b_{i,j}$ and by $\theta(b_{i,j}^{(\ell)})$ the slope of $b_{i,j}^{(\ell)}$.
\item The broken lines $b_{i,j}$ do not intersect each other and they are disjoint from the singularities of $\eta$.
 \item The slopes $\theta(b_{i,j}^{(\ell)})$ are different from that of $\pm\zeta^{jk/d} \tilde{r}_{i,j}$ for all $j\in\left\{0,\dots,\frac{k}{d}-1\right\}$.
 \item The holonomy $h(b_{i,j})$ associated to $b_{i,j}$ is $\zeta^{jk/d}$ with respect to the $(1/k)$-translation structure of $X_0$.
\end{itemize}
\par
For $t\in (0, \varepsilon)$ sufficiently small, we define the surface $X_{0,t} = X_{0,t}(T,\pi)$
obtained by modifying neighborhoods of size $tT$ around the residue slit as follows. Remove $P(T,\pi)$
from $X_0$. In the case that $\langle h(b_{i,j}^{(\ell)})^{-1} \theta(b_{i,j}^{(\ell)}), \tilde{r}_{i,j} \rangle > 0$, we remove neighborhood parallelograms swept out by planar segments with holonomy vector~$t\cdot h(b_{i,j}^{(\ell)})\tilde{r}_{i,j}$ centered along $b_{i,j}^{(\ell)}$ in general, and glue the parallel sides of these removed parallelograms in pairs. On the other hand for each of the segments where $\langle h(b_{i,j}^{(\ell)})^{-1}\theta(b_{i,j}^{(\ell)}), \tilde{r}_{i,j} \rangle < 0$, we add a parallelogram swept out by segments of holonomy~$t\cdot h(b_{i,j}^{(\ell)})\tilde{r}_{i,j}$ to the existing surface and glue the pieces as indicated in \cite[Figure~16]{plumb}.
\par
 \begin{figure}[ht]
 \centering
\begin{tikzpicture}[scale=1.5, 
  ->-/.style={%
      postaction={decorate},decoration={%
          markings,mark=at position #1 with {\arrow{stealth}}%
      }%
  },->-/.default=.5,
   -<-/.style={
      postaction={decorate},decoration={%
          markings,mark=at position #1 with {\arrowreversed{stealth}}%
      }%
  },-<-/.default=.5]

 \begin{scope}
 \begin{scope}
 \clip  (-2,-2)  --  ++(4,0) -- ++(0,4)-- ++(-4,0)--   ++(0,-4);
    \fill[ color=black!10!] (0,0) circle (3);
\end{scope}
 \draw  (-2,-2)  -- coordinate[pos=.5](a1) ++(4,0) -- coordinate[pos=.5](b1)  ++(0,4)-- coordinate[pos=.5](a2) ++(-4,0)-- coordinate[pos=.5](b2)  ++(0,-4);
\node[below] at (a1) {$1$};
\node[right] at (b1) {$2$};
\node[above] at (a2) {$1$};
\node[left] at (b2) {$2$};

\draw (-1,0.5) coordinate (c0) --coordinate[pos=.5](c1)coordinate[pos=.833](c2)coordinate[pos=0.166](c3)coordinate[pos=1] (c4) ++(120:1);
\draw (1,0.5) coordinate (d0) --coordinate[pos=.5](d1)coordinate[pos=.833](d2)coordinate[pos=0.166](d3)coordinate[pos=1] (d4) ++(60:1);
\node[below,rotate=-60] at (c1) {$3$};
\node[below,rotate=120] at (c1) {$4$};
\node[xshift=-.2cm,yshift=.05cm,rotate=60] at (d1) {$3$};
\node[above,rotate=-120] at (d1) {$4$};

 \fill (c0) circle (2pt);        \fill (d4) circle (2pt);
\fill[white] (c4) circle (2pt);        \fill[white] (d0) circle (2pt);
\draw (c4) circle (2pt);        \draw (d0) circle (2pt);

\fill (0,-1) circle (1pt); \node[below] at (0,-1) {$q_{0}$};
\fill (-1/3,-1) circle (1pt); \node[below] at (-1/3,-1) {$q_{1}$};
\fill (1/3,-1) circle (1pt); \node[below] at (1/3,-1) {$q_{-1}$};
\draw[thin] (-1/2,-1) -- (1/2,-1);

\draw[dashed] (0,.405)-- ++(240:1/3) coordinate[pos=.5](t1)-- ++(0:1/3) coordinate[pos=.5](t2) -- (0,.405)  coordinate[pos=.5](t3);

\draw[dotted] (-1/3,-1) -- ++ (90:.1) -- ++ (160:.976) -- (c1) -- (t1);
\draw[dotted] (0,-1) --  (t2);
\draw[dotted] (1/3,-1) -- ++ (90:.1) -- ++ (20:.976) -- (d1) -- (t3);

 \end{scope}

 \begin{scope}[xshift=5cm]
   \fill[ color=black!10!] (0,-0.5) circle (1.4);

    \fill (0,-1) circle (2pt);

\draw (0,-1) --coordinate[pos=.5](e1) ++(150:1/1.711)coordinate[](f1);
\draw (0,-1) --coordinate[pos=.5](e2) ++(30:1/1.711)coordinate[](g1);
\draw (f1) -- ++(90:1.4)coordinate[pos=.5](e3);
\draw (g1) -- ++(90:1.4)coordinate[pos=.5](e4);

\fill[white] (g1) circle (2pt);        \fill[white] (f1) circle (2pt);
\draw (g1) circle (2pt);        \draw (f1) circle (2pt);

     \fill[color=white]
    (0,-1) --  ++(30:1/1.711) --++(90:1.7)  -- ++(180:1)--++(-90:1.7) -- cycle;
    \draw (0,-1) --coordinate[pos=.6](e1) ++(150:1/1.711)coordinate[](f1);
\draw (0,-1) --coordinate[pos=.6](e2) ++(30:1/1.711)coordinate[](g1);
\draw (f1) -- ++(90:1.4)coordinate[pos=.5](e3);
\draw (g1) -- ++(90:1.4)coordinate[pos=.5](e4);

\node[left] at (e3) {$5$};
\node[right] at (e4) {$5$};
\node[left,rotate=60] at (e1) {$6$};
\node[right,rotate=-60] at (e2) {$6$};
 \end{scope}

\end{tikzpicture}

 \caption{Flat geometric representation of a twisted cubic differential and the residue slit}
\label{fig:triplepres}
\end{figure}
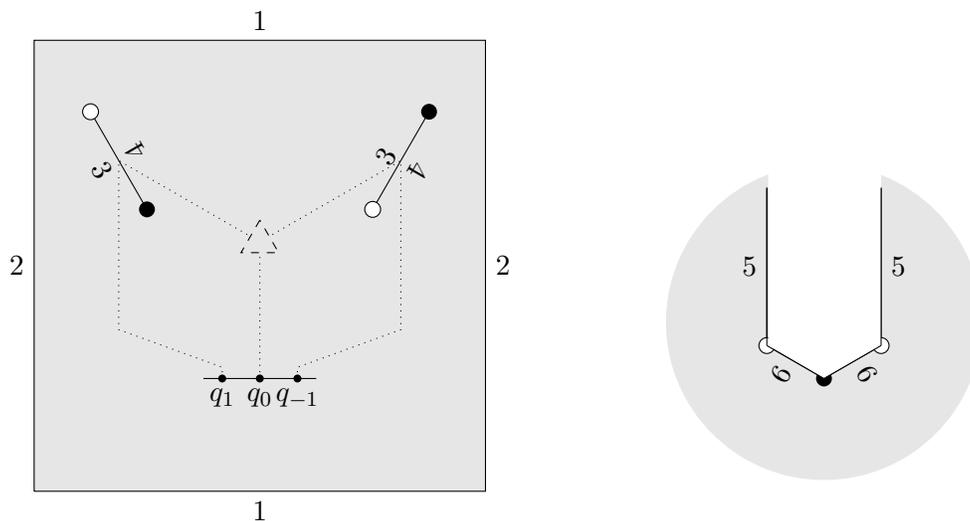
Let us carry out the above construction in a simple but already interesting example. Consider the stratum $\Omega^{3}\moduli[2](3,3,1,-1)$, where the twisted cubic differential $(X, \eta)$ is given by $(X_{0},\eta_{0})\in\Omega^{3}\moduli[2](3,3,1,0,-1)$  and $(X_{1},\eta_{1})\in\Omega^{3}\moduli[0](1,-1,-6)$ such that a marked ordinary point $q$ of $X_{0}$ is attached to the pole of order $6$ of $X_{1}$. In particular, this example illustrates the operation of breaking up a regular point into a simple zero and a simple pole for cubic differentials. We present $\eta_0$ and $\eta_1$ by using $(1/3)$-translation structure along with the residue slit in Figure~\ref{fig:triplepres}.

 \begin{figure}[ht]
 \centering
\begin{tikzpicture}[scale=2, 
  ->-/.style={%
      postaction={decorate},decoration={%
          markings,mark=at position #1 with {\arrow{stealth}}%
      }%
  },->-/.default=.5,
   -<-/.style={
      postaction={decorate},decoration={%
          markings,mark=at position #1 with {\arrowreversed{stealth}}%
      }%
  },-<-/.default=.5]

 \begin{scope}
 \begin{scope}
 \clip  (-2,-2)  --  ++(4,0) -- ++(0,4)-- ++(-4,0)--   ++(0,-4);
    \fill[ color=black!10!] (0,0) circle (3);
\end{scope}
 \draw  (-2,-2)  -- coordinate[pos=.5](a1) ++(4,0) -- coordinate[pos=.5](b1)  ++(0,4)-- coordinate[pos=.5](a2) ++(-4,0)-- coordinate[pos=.5](b2)  ++(0,-4);

\draw (-1,0.5) --coordinate[pos=0](c0)coordinate[pos=.5](c1)coordinate[pos=.833](c2)coordinate[pos=0.166](c3)coordinate[pos=1](c4) ++(120:1);
\draw (1,0.5) --coordinate[pos=0](d0)coordinate[pos=.5](d1)coordinate[pos=.833](d2)coordinate[pos=0.166](d3)coordinate[pos=1](d4) ++(60:1);

\draw (0,-1.5) --coordinate[pos=.5](e1) ++(150:1/1.711)coordinate[](f1);
\draw (0,-1.5) --coordinate[pos=.5](e2) ++(30:1/1.711)coordinate[](g1);

\draw (0,.405) coordinate(t1)-- ++(240:1/3) coordinate(t2)-- ++(0:1/3) coordinate(t3)-- cycle;

\draw[thin] (f1)  -- ++ (90:.1) coordinate(r1) -- ++ (160:.976) coordinate(s1) -- (c2);
\draw[thin] (g1)  -- ++ (90:.1) -- ++ (20:.976) coordinate(s2) -- (d2);

\draw[dashed] (r1) -- ++(1/3,0)coordinate(r2)-- ++(1/3,0)coordinate(r3)-- ++(1/3,0)coordinate(r4);

\draw[thin] (t2) -- (r2)  -- ++ (160:.976) coordinate (s3) -- (c3);
\draw[thin] (t3) -- (r3)  -- ++ (20:.976)coordinate (s4) -- (d3);

    \fill (0,-1.5) circle (1pt);
    \fill[white] (g1) circle (1pt);        \fill[white] (f1) circle (1pt);
\draw (g1) circle (1pt);        \draw (f1) circle (1pt);

   \fill[color=white]
    (0,-1.5) --  (f1) --(r1)  -- (s1) -- (c2) -- (t1) -- (d2) -- (s2) -- (r4) --(g1) -- cycle;
    \fill[ color=black!10!] (r2) -- (s3)-- (c3) -- (t2) -- cycle;
    \fill[ color=black!10!] (r3) -- (s4) -- (d3) -- (t3) -- cycle;

 \fill (c0) circle (1pt);        \fill (d4) circle (1pt);
\fill[white] (c4) circle (1pt);        \fill[white] (d0) circle (1pt);
\draw (c4) circle (1pt);        \draw (d0) circle (1pt);

\node[left] at (f1) {$5'$};
\node[right] at (g1) {$5'$};
\node[left,rotate=60] at (e1) {$6$};
\node[right,rotate=-60] at (e2) {$6$};

\draw (-1,0.5) --coordinate[pos=.5](c1)coordinate[pos=.833](c2)coordinate[pos=0.166](c3) ++(120:1);
\draw (1,0.5) --coordinate[pos=.5](d1)coordinate[pos=.833](d2)coordinate[pos=0.166](d3) ++(60:1);

\draw (0,-1.5) --coordinate[pos=.5](e1) ++(150:1/1.711)coordinate[](f1);
\draw (0,-1.5) --coordinate[pos=.5](e2) ++(30:1/1.711)coordinate[](g1);

\draw[dashed] (0,.405) coordinate(t1)-- ++(240:1/3) coordinate(t2)-- ++(0:1/3) coordinate(t3)-- cycle;

\draw[thin] (f1)  -- ++ (90:.1) coordinate(r1) -- ++ (160:.976) coordinate(s1) --(c2) coordinate[pos=.8](l1);
\draw[thin] (g1)  -- ++ (90:.1) -- ++ (20:.976) coordinate(s2) -- (d2) coordinate[pos=.8](l2);

\draw[dashed] (r1) -- ++(1/3,0)coordinate(r2)-- ++(1/3,0)coordinate(r3)-- ++(1/3,0)coordinate(r4);
\draw[dashed] (f1) -- (g1);
\draw[dashed] (s1) -- (s3);
\draw[dashed] (c3) -- ++(0:-1/3);
\draw[dashed] (c3) -- ++(60:1/3);
\draw[dashed] (s2) -- (s4);
\draw[dashed] (d3) -- ++(0:1/3);
\draw[dashed] (d3) -- ++(120:1/3);

\draw[thin] (t2) -- (r2) coordinate[pos=.5](p5)  -- ++ (160:.976) -- (c3);
\draw[thin] (t3) -- (r3) coordinate[pos=.5](q5) -- ++ (20:.976) -- (d3);

\node[left] at (p5) {$5''$};
\node[right] at (q5) {$5''$};

\draw[thin] (d2)  -- (t1)  coordinate[pos=.2](l3);
\draw[thin] (d3)  -- (t3);
\draw[thin] (c2)  -- (t1)  coordinate[pos=.2](l4);
\draw[thin] (c3)  -- (t2);

\node[left] at (l1) {$1$};
\node[right,rotate=120] at (l3) {$1$};
\node[right] at (l2) {$2$};
\node[left,rotate=-120] at (l4) {$2$};

 \end{scope}
\end{tikzpicture}

 \caption{A plumbed surface in $\Omega^{3}\moduli[2](3,3,1,-1)$}
\label{fig:tripleresidue}
\end{figure}
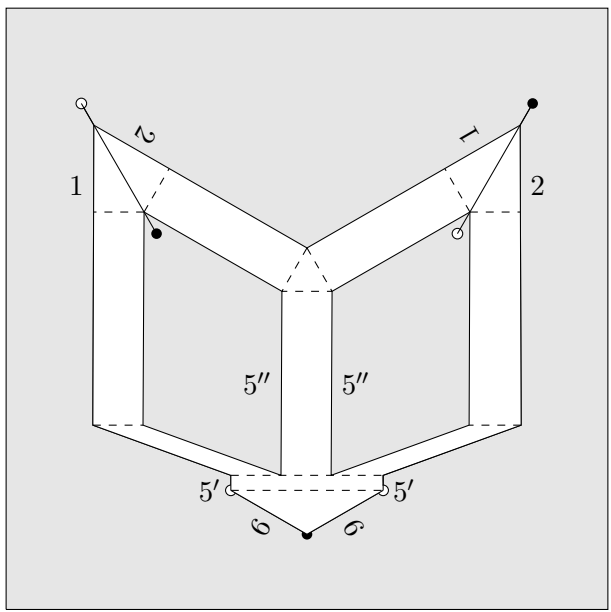
In Figure~\ref{fig:triplepres}, the point $q$ coincide with $q_{0}$ and the segment representing the $k$-residue is the horizontal one. The polygon $P(T,\pi)$ is the equilateral triangle in dashed lines. The residue slit is pictured in dotted lines.
The cubic differential obtained after gluing in the lower level surface $t\eta_1$ at $q$ and opening up the residue slit is presented in Figure~\ref{fig:tripleresidue}, where the unlabeled edge identifications are clear and the identifications denoted by $1$ and $2$ are given by composing translation with a rotation of angle $2\pi/3$. As $t \to 0$, i.e., if we shrink the removed residue slit neighborhood, the surface goes to $(X_0, \eta_0)$. Up to scaling, if alternatively we expand everything arbitrarily large compared to the residue slit neighborhood, we then obtain $(X_1, \eta_1)$. A rigorous proof that justifies convergence in the sense of the incidence variety compactification can be found in~\cite[Paragraph~5.3]{plumb}.

\end{exa}


\printbibliography
\end{document}